\documentclass[a4paper]{amsart}
\usepackage{graphicx}
\theoremstyle{plain}
  \newtheorem{thm}{Theorem}[section]
  \newtheorem{lem}[thm]{Lemma}
  \newtheorem{cor}[thm]{Corollary}
  \newtheorem{prop}[thm]{Proposition}

\theoremstyle{definition}
  
  \newtheorem{ex}[thm]{Example}

\theoremstyle{remark}
  
  \newtheorem*{ack}{Acknowledgments}
\newcommand{\Z}{\mathbb{Z}}
\newcommand{\Zpos}{\mathbb{Z}_{> 0}}

\newcommand{\cyclic}[1]{\Z/#1\Z}
\newcommand{\cover}[2]{#1^{(#2)}}
\newcommand{\baseSet}[1]{B_{#1}}
\newcommand{\nanoword}[2]{#1\colon\hspace{-1mm}#2}
\newcommand{\link}[2]{l(#1,#2)}
\newcommand{\linkpattern}[4]{\dotso #1 \dotso #2 \dotso #3 \dotso #4 \dotso}
\newcommand{\trivial}{0}
\newcommand{\xType}[1]{\lvert#1\rvert}
\newcommand{\abs}[1]{\lvert#1\rvert}
\newcommand{\setStrings}{\mathcal{VS}}
\newcommand{\transpose}[1]{{}^t#1}
\newcommand{\cvector}[1]{\overrightarrow{#1}}
\newcommand{\cable}[2]{#1_{(#2)}}
\newcommand{\textcable}[1]{$#1$-cable}
\newcommand{\textcover}[1]{$#1$-covering}

\DeclareMathOperator{\rank}{rank}
\DeclareMathOperator{\hr}{hr}
\DeclareMathOperator{\height}{height}
\DeclareMathOperator{\base}{base}
\DeclareMathOperator{\n}{n}
\DeclareMathOperator{\m}{m}

\numberwithin{equation}{section}
\begin{document}
\title{Coverings, Composites and Cables of Virtual Strings}
\author{Andrew Gibson}
\address{
Department of Mathematics,
Tokyo Institute of Technology,
Oh-okayama, Meguro, Tokyo 152-8551, Japan
}
\email{gibson@math.titech.ac.jp}
\date{\today}
\begin{abstract}
A virtual string can be defined as an equivalence class of planar
 diagrams under certain kinds of diagrammatic moves. Virtual strings are
 related to virtual knots in that a simple operation on a virtual knot
 diagram produces a diagram for a virtual string.
\par
In this paper we consider three operations on a virtual string or
 virtual strings which produce another virtual string, namely covering,
 composition and cabling.
In particular we study virtual strings unchanged by the covering
 operation. 
We also show how the based matrix of a composite virtual string
 is related to the based matrices of its components, correcting a result
 by Turaev.
Finally we investigate what happens under cabling to some
 invariants defined by Turaev.
\end{abstract}
\keywords{virtual strings, virtual knots}
\subjclass[2000]{Primary 57M25; Secondary 57M99}
\thanks{This work was supported by a Scholarship from the Ministry of
Education, Culture, Sports, Science and Technology of Japan. Most of
the contents of this paper come from part of the author's Master's
thesis submitted to the Tokyo Institute of Technology in March 2008
\cite{Gibson:mthesis}.} 
\maketitle
\section{Introduction}
Kauffman introduced the idea of virtual knots in
\cite{Kauffman:VirtualKnotTheory}. Virtual knots are defined as
equivalence classes of virtual knot diagrams under diagrammatic moves
which include the usual Reidemeister moves and other similar moves
involving virtual crossings.
\par
A virtual string diagram is a virtual knot diagram where the over and
under crossing information at the real crossings has been removed.
In other words, the real crossings are treated simply as double points. 
We call this operation \emph{flattening}.
By flattening the diagrammatic
moves given by Kauffman we can derive moves for virtual string
diagrams. A virtual string is then an equivalence class of virtual
string diagrams under these moves. A complete definition will be given
in Section~\ref{sec:virtual_strings}.
In some parts of the literature virtual strings are known by other
names. For example in \cite{Hrencecin/Kauffman:Filamentations} they are
called flat knots or flat virtual knots, in
\cite{Kadokami:Non-triviality} projected virtual knots, and in
\cite{CKS:StableEquivalence} universes of virtual knots.
\par
In \cite{Turaev:2004}, Turaev defines virtual strings in terms of
diagrams consisting of a circle and a finite set of ordered pairs of
distinct points on the circle. This definition can be shown to be
equivalent to the definition given above. In that paper, Turaev defines
the $u$-polynomial and the primitive based matrix which are invariants of
virtual strings. We recall these definitions in this paper.
\par
Two virtual knot diagrams representing the same virtual knot are related
by a sequence of moves. By flattening the moves we get a sequence of
flattened moves relating the corresponding flattened diagrams.
From this we can see that the virtual string derived from
the flattening of a particular virtual knot diagram is actually an
invariant of the virtual knot which the diagram represents. Of course,
many virtual knots may have the same underlying virtual string. For
example, the virtual string underlying every classical knot is the
trivial virtual string. 
On the other hand, the virtual string underlying Kishino's knot
(Figure~\ref{fig:kishino}) is not trivial
\cite{Fenn/Turaev:WeylAlgebras} and this shows that Kishino's knot is
not a trivial virtual knot.  
Kishino, using a different method, was the first to prove the
non-triviality of this virtual knot \cite{Kishino/Shin:VirtualKnots}.
Various other methods of proof have been found and these are summarised
in Problem~1 of the list of problems in \cite{FKM:Unsolved}.
However, we note that the method used by Kadokami to prove
non-triviality of the virtual string underlying Kishino's knot
\cite{Kadokami:Non-triviality} is based on a theorem in that paper with
which we found a problem. This problem is explained in
\cite{Gibson:tabulating-vs}.
\begin{figure}[hbt]
\begin{center}
\includegraphics{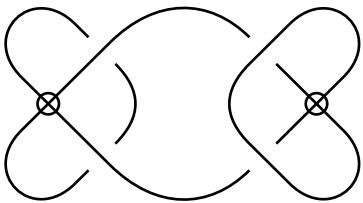}
\caption{A virtual knot known as Kishino's knot}
\label{fig:kishino}
\end{center}
\end{figure}
\par
In this paper we study three different operations on virtual strings
which produce new virtual strings.
\par
In \cite{Turaev:2004}, for each non-negative integer $r$, Turaev defined
an operation on a virtual string called an $r$-covering which produces
another virtual string.
Thus we can consider an $r$-covering to be a map from the set of virtual
strings to itself.
The result of an $r$-covering is an invariant of the original virtual
string \cite{Turaev:2004}.
Thus invariants of the new virtual string may be considered as
invariants of the original virtual string.
Of course, it is possible to take the covering of the new virtual string
and get a hierarchy of virtual strings derived from the original one.
We study certain questions about this operation.
We show that $r$-covering is surjective for all $r$ and when $r$ is not
$1$, there are an infinite number of virtual strings that map to any
given virtual string under $r$-covering.
We also show that for all $r$ the set of virtual strings unchanged by
the $r$-covering operation is infinite.
\par
Given two virtual string diagrams we can make a composite virtual string
diagram by cutting the curve in each diagram and joining them to each
other to make a single curve. This operation is not well-defined for
virtual strings as the result depends on the points where the curves in
the original diagrams were cut. However, it is still possible to examine
how the invariants of virtual strings created in this way are related to
the invariants of the virtual strings from which they are constructed.
\par
For classical knots, invariants of cables of knots can be used to
distinguish knots where the same invariants calculated directly on the knots
themselves are the same. We define a cabling of a virtual string and
study what happens to Turaev's invariants under this operation. We
discovered that we do not gain any more information from Turaev's
invariants in this way.
We show how Turaev's invariants for the cable of a virtual string can be
calculated from the same invariant for the virtual string itself.
\par
In Section~\ref{sec:virtual_strings} we give a formal definition of
virtual strings. In this paper we will often use Turaev's nanoword notation
\cite{Turaev:KnotsAndWords} to represent virtual strings. This notation
is explained in Section~\ref{sec:nanowords}. In the same section we also
recall the definition of Turaev's $u$-polynomial.
\par
In Section~\ref{sec:th_matrices} we define the head and tail matrices of
a virtual string diagram. We use these when we recall the definition of
Turaev's based matrices in Section~\ref{sec:based_matrices}. From a
based matrix we can derive a primitive based matrix which is another
invariant of virtual strings.
\par
In Section~\ref{sec:coverings} we recall the definition of covering,
make some observations about the operation and define some invariants of
virtual strings from it. Then in Section~\ref{sec:fixed} we consider
fixed points under covering. In Section~\ref{sec:geomcover} we explain a
geometric intrepretation of the covering operation.
\par
In Section~\ref{sec:composite}, we show how the based matrix of a 
composite virtual string is related to the based matrices of its components. 
\par
In Section~\ref{sec:cable} we consider cables of virtual strings. We
show that the $u$-polynomial of a cable can be calculated directly from
the $u$-polynomial of the original virtual string. We also show a
corresponding result for based matrices. In the same section we also
show a relationship between coverings of cables and cables of
coverings. We conclude that if we have a pair of virtual strings with
the same $u$-polynomial, based matrix and coverings then the
corresponding invariants of an \textcable{n} of each virtual string will
also be equivalent.
\begin{ack}
The author would like to thank his supervisor Hitoshi Murakami for all
 his help and advice.
\end{ack}
\section{Virtual strings}\label{sec:virtual_strings}
A virtual string diagram is an oriented circle immersed in a
plane. Self-intersections are permitted but at most two arcs can cross
at any particular point and they should cross transversally. We call
such self-intersections crossings and we allow two kinds: real
crossings, which are unmarked; and virtual crossings, each of which is
marked with a small circle (see Figure~\ref{fig:virtualcrossings}).
An example of a virtual string diagram is given in
Figure~\ref{fig:diagram31}, where the orientation of the circle is
marked by an 
arrow.
\begin{figure}[hbt]
\begin{center}
\includegraphics{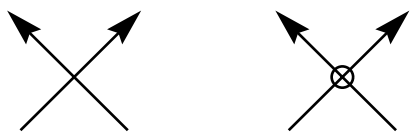}
\caption{The two kinds of crossing: real (left) and virtual (right)}
\label{fig:virtualcrossings}
\end{center}
\end{figure}
\begin{figure}[hbt]
\begin{center}
\includegraphics{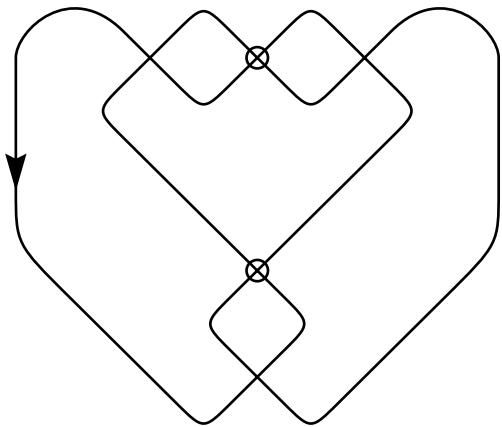}
\caption{A non-trivial virtual string}
\label{fig:diagram31}
\end{center}
\end{figure}
\par
Moves have been defined for virtual string diagrams
\cite{Kauffman:VirtualKnotTheory}.
There are moves involving only
real crossings which are shown in Figure~\ref{fig:reidermeister}. These
are called the flattened Reidemeister moves because they are like the
standard Reidemeister moves of knot theory but with crossings
flattened (see, for example, \cite{Lickorish:1997} for more information
about standard Reidemeister moves). 
There are a similar set of moves only involving virtual crossings. We call
these the virtual flattened Reidemeister moves and they are shown in
Figure~\ref{fig:vreidermeister}. Lastly there is a move that involves
both real and virtual crossings. It is shown in
Figure~\ref{fig:mixedmove} and we call it the mixed move. Collectively,
flattened Reidemeister moves, virtual flattened Reidemeister moves and
the mixed move are called homotopy moves.
\par
\begin{figure}[hbt]
\begin{center}
\includegraphics{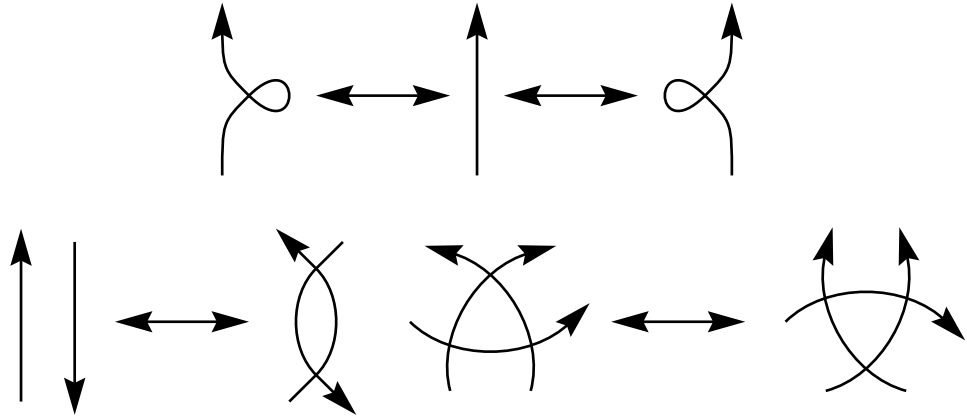}
\caption{The flattened Reidemeister moves}
\label{fig:reidermeister}
\end{center}
\end{figure}
\begin{figure}[hbt]
\begin{center}
\includegraphics{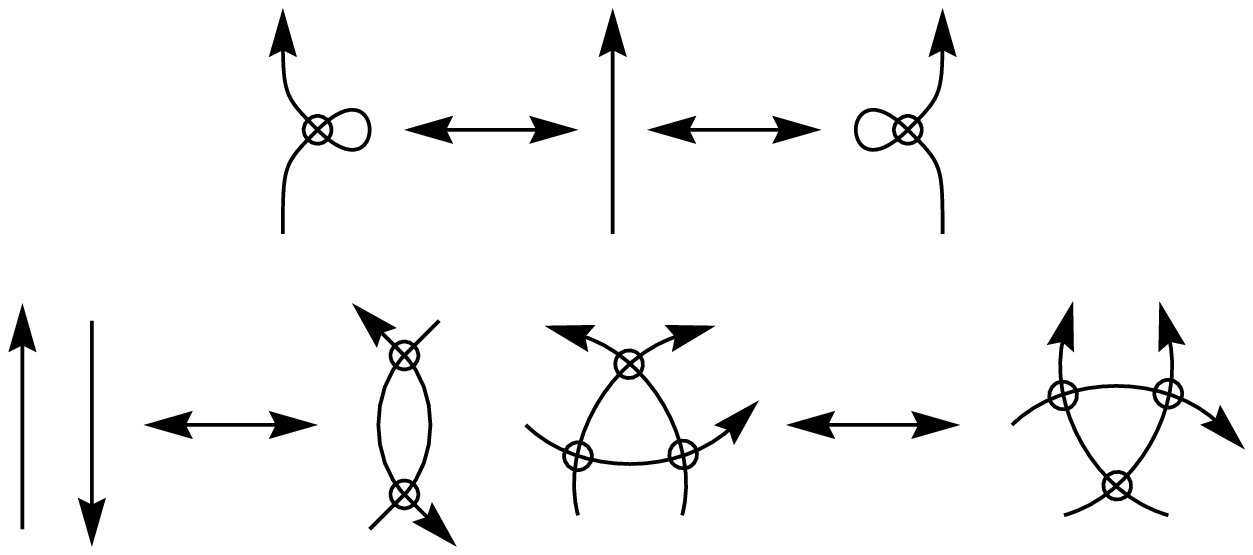}
\caption{The virtual flattened Reidemeister moves}
\label{fig:vreidermeister}
\end{center}
\end{figure}
\begin{figure}[hbt]
\begin{center}
\includegraphics{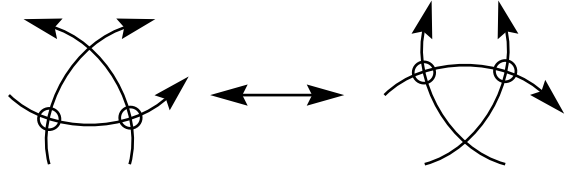}
\caption{The mixed move}
\label{fig:mixedmove}
\end{center}
\end{figure}
If a pair of virtual string diagrams are related by a finite sequence of
homotopy moves and ambient isotopies in the plane, we say that they are
equivalent under homotopy, or, more 
simply, that they are homotopic. It is not hard to see that
equivalence under homotopy is an equivalence relation. Virtual
strings are the equivalence classes of virtual string diagrams under
this relation. We say that a 
virtual string $\Gamma$ is represented by a particular diagram $D$ if
the equivalence class under homotopy containing $D$ is $\Gamma$.
\par
Let $\setStrings$ denote the set of equivalence classes, under homotopy,
of virtual strings that can be represented by a diagram with a finite
number of double points. In this paper we only consider virtual strings
that are in $\setStrings$.
\par
There is a unique virtual string represented by a diagram with no
double points, real or virtual. It is called the trivial virtual string
and is written $\trivial$.
\par
We can think of virtual strings in another way. We consider pairs
$(S,D)$ where $S$ is a compact oriented surface and $D$ is an
immersion of an oriented circle in $S$. As before we only allow
self-intersections in $D$ to be transverse double points. This time all
such crossings are real. Virtual crossings are not permitted.
\par
For a pair $(S,D)$, we define $N(D)$ to be the regular
neighbourhood of $D$ in $S$. A stable homeomorphism from a pair
$(S_1,D_1)$ to a pair $(S_2,D_2)$ is an orientation preserving
homeomorphism from $N(D_1)$ to $N(D_2)$. Here orientation preserving
means both the orientation of the surface and of the curve itself are
preserved. Two pairs are stably equivalent if there exists a finite
sequence of stable homeomorphisms and flattened Reidemeister moves
in the surface transforming one pair to the other.
\par
Clearly stable equivalence is an equivalence relation. There is a
bijection between the set of equivalence classes under this relation and
the set of virtual strings. This was shown by Kadokami in
\cite{Kadokami:Non-triviality}, following a result of Carter, Kamada and
Saito which relates virtual knots to non-virtual diagrams on oriented
surfaces \cite{CKS:StableEquivalence}.
The bijection can be visualized as follows.
From a virtual string diagram we can construct a
pair $(S,D)$ by replacing virtual crossings with handles in the plane
and routing one arc over the handle (see Figure~\ref{fig:handle}).
On the other hand, with some care, we can take a pair $(S,D)$ and project
$D$ onto a plane so that the only self-intersection points are double
points. Any double points that do not correspond to double points in $D$
are marked as virtual. The result is a virtual string diagram.
\begin{figure}[hbt]
\begin{center}
\includegraphics{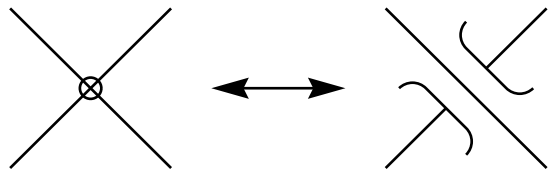}
\caption{Changing a surface with a virtual crossing (left) to a
 surface with a hollow handle and no crossing (right)}
\label{fig:handle}
\end{center}
\end{figure}
\par
The \emph{canonical surface} of a virtual string diagram $D$ is a
surface of minimal genus containing $D$. It is unique up to
homeomorphism of the surface. Such a surface can be easily
constructed. First construct a surface containing $D$ using the process
described above. Then cut $N(D)$ from the surface and construct a new
surface by glueing a disk to each boundary of $N(D)$. The result is the
canonical surface of $D$. This construction is described in several
places, for example \cite{Kadokami:Non-triviality} or
\cite{Turaev:2004}.
\par
If we allow diagrams with multiple oriented circles then the equivalence
classes of these diagrams under the flattened Reidemeister moves form a
generalization of virtual strings. We call these multi-component
virtual strings. This generalization is analogous to the generalization
of virtual knots to virtual links. Indeed, multi-component virtual
strings can be viewed as flattened virtual links. In general we only
consider single component virtual strings in this paper.
However, multi-component virtual strings appear in
Section~\ref{sec:geomcover}.
\section{Nanowords}\label{sec:nanowords}
In \cite{Turaev:Words}, Turaev defined the concept of a nanoword. We recall the
definition here.
\par
A \emph{word} is an ordered sequence of elements of a set. We call the
elements of the set \emph{letters}. A \emph{Gauss word} is a word where
each letter appears exactly twice or not at all.
For some set fixed set $\psi$, a \emph{nanoword over $\psi$} is defined
to be a Gauss word with a map from the set of letters appearing in the
Gauss word to $\psi$ \cite{Turaev:Words}.
\par
In \cite{Turaev:KnotsAndWords} Turaev showed that we can represent a
virtual string diagram as a nanoword over the set $\{a,b\}$. As we will
only be interested in this kind of nanoword in this paper, from now on
we will write \emph{nanoword} to mean a nanoword over $\{a,b\}$. We now
explain how to associate such a nanoword to a virtual string diagram.
\par
Given a virtual string diagram, we label the real crossings and
introduce a base point on the curve at some point other than a
crossing. Starting at the base point, we follow the curve according to
its orientation and record the 
labels of the crossings as we pass through them. When we get back to the
base point we have passed through each crossing exactly twice and the
sequence of labels we have recorded is a Gauss word.
We assign a map from the letters in the Gauss word to the set
$\{a,b\}$ in the following way.
For a given letter we consider its corresponding crossing. If, during
our traversal of the curve, the second time we passed through the
crossing, we crossed the first arc from right to left, we map the letter
to $a$. If we crossed the first arc from left to right, we map the
letter to $b$. These two cases are shown in Figure~\ref{fig:crossing}.
\par
\begin{figure}[hbt]
\begin{center}
\includegraphics{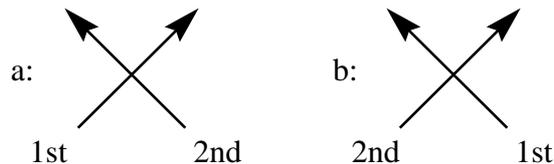}
\caption{The two types of real crossing}
\label{fig:crossing}
\end{center}
\end{figure}
\begin{figure}[hbt]
\begin{center}
\includegraphics{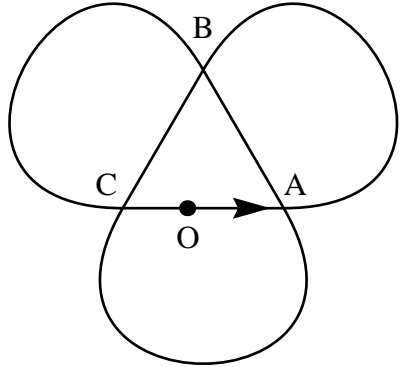}
\caption{The flattened trefoil with base point and crossing points labelled}
\label{fig:trefoil}
\end{center}
\end{figure}
As an example we calculate the nanoword corresponding to the diagram of
a flattened trefoil shown in Figure~\ref{fig:trefoil}. Starting at the
base point $O$, traversing the curve gives the Gauss word $ABCABC$. By
comparing each crossing 
to those in Figure~\ref{fig:crossing} we define the map from to
$\{A,B,C\}$ to $\{a,b\}$. In this case $A$ and $C$ map to $a$ and $B$
maps to $b$. 
\par
As $a$ and $b$ encode the crossing type we sometimes refer
to them as types. Thus in our example the type of $A$ is $a$ and the
type of $B$ is $b$.
Following Turaev \cite{Turaev:Words}, we use the notation $\xType{X}$ to
mean
the type of the letter $X$. So, in this example $\xType{A}$ is $a$,
$\xType{B}$ is $b$ and $\xType{C}$ is $a$.
\par
The map from the letters to the types can be
represented in a compact form by listing, in alphabetical order of the
letters, the images under the map. So in this example we can represent
the map by $aba$ which expresses the fact that $A$ maps to $a$, $B$ maps
to $b$ and $C$ maps to $a$. The nanoword from our example can then be
written simply as $\nanoword{ABCABC}{aba}$. We will use this format
often in this paper. 
\par
It is sometimes useful to draw an arrow diagram of a nanoword. Here we
write the letters of the Gauss word in order and then join each pair of
identical letters by an arrow. The direction of the arrow indicates the
crossing type. If the arrow goes from left to right, the crossing is a
type $a$ crossing. If the arrow goes from right to left, the crossing is
a type $b$ crossing. As an example, an arrow diagram of the nanoword
$\nanoword{ABCABC}{aba}$ is shown in Figure~\ref{fig:abcabc}.
\begin{figure}[hbt]
\begin{center}
\includegraphics{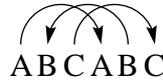}
\caption{Arrow diagram of nanoword $\nanoword{ABCABC}{aba}$}
\label{fig:abcabc}
\end{center}
\end{figure}
\par
The rank of a nanoword $\alpha$ is the number of different letters in
the Gauss word \cite{Turaev:Words}.
This is the number of real crossings in the virtual
string diagram which $\alpha$ represents. We write the rank of $\alpha$
as $\rank(\alpha)$.
In the previous example the rank of the nanoword is
$3$.
\par
An isomorphism of nanowords \cite{Turaev:Words} is a bijection $i$ from
the letters of a nanoword 
$\alpha_1$ to the letters of another nanoword $\alpha_2$ which satisfies the
following two requirements. Firstly, that $i$ maps the $n$th letter of
$\alpha_1$ to the $n$th letter of $\alpha_2$ for all $n$. Secondly, for
each letter $X$ in $\alpha_1$, $\xType{X}$ is equal to $\xType{i(X)}$.
Two nanowords $\alpha_1$ and $\alpha_2$ are isomorphic if there exists such an
isomorphism between them. Diagrammatically, if we relabel the real
crossings of a diagram, the nanowords associated with the diagram before
and after the relabelling will be isomorphic.
\par
Note that the nanoword representation is dependent on the base point that we
picked. To remove dependence on the base point Turaev defined a shift
move \cite{Turaev:KnotsAndWords}.
\par
A shift move takes the first letter in the nanoword and moves it to
the end of the nanoword. In the new nanoword, the moved letter is mapped to the
opposite type. The inverse of the shift move takes the last letter in the
nanoword and moves it to the beginning of the nanoword. Again, the type of the
moved letter is swapped. We can write the move like this:
\begin{equation*}
AxAy \longleftrightarrow xA^\prime yA^\prime
\end{equation*}
where $A$ and $A^\prime$ are arbitrary letters which map to opposite
types and $x$ and $y$ represent arbitrary sequences of letters such that
the words on either side of the move are Gauss words.
\par
Turaev defined homotopy moves for nanowords over any set
\cite{Turaev:Words}.
Here we
describe the moves in the specific case of nanowords over $\{a,b\}$.
When we describe moves on nanowords we use the following
conventions. Arbitrary individual letters are represented by upper case
letters. Lower case letters $x$, $y$, $z$ and $t$ are used to represent
sequences of letters. These sequences are arbitrary under the
constraint that each side of a move should be a Gauss word. 
\par
Homotopy move 1 (H1):
\begin{equation*}
xAAy \longleftrightarrow xy
\end{equation*}
where $\xType{A}$ is either $a$ or $b$.
\par
Homotopy move 2 (H2):
\begin{equation*}
xAByBAz \longleftrightarrow xyz
\end{equation*}
where $\xType{A}$ is not equal to $\xType{B}$.
\par
Homotopy move 3 (H3):
\begin{equation*}
xAByACzBCt \longleftrightarrow xBAyCAzCBt
\end{equation*}
where $\xType{A}$, $\xType{B}$ and $\xType{C}$ are all the same.
\par
Note that these moves correspond to flattened Reidemeister moves.
A detailed explanation of this correspondance is given in
\cite{Gibson:mthesis}.
\par
Turaev derived some simple moves from moves H1, H2
and H3. They appear in Lemmas~3.2.1 and 3.2.2 in \cite{Turaev:Words}.
We quote them here:
\begin{equation*}
\begin{array}{lll}
\textrm{H2a:}\quad & xAByABz \longleftrightarrow xyz & 
\textrm{where $\xType{A}\ne\xType{B}$}, \\
\textrm{H3a:}\quad & xAByCAzBCt \longleftrightarrow xBAyACzCBt &
\textrm{where $\xType{A}=\xType{C}\ne\xType{B}$}, \\
\textrm{H3b:}\quad & xAByCAzCBt \longleftrightarrow xBAyACzBCt & 
\textrm{where $\xType{A}=\xType{B}\ne\xType{C}$}, \\
\textrm{H3c:}\quad & xAByACzCBt \longleftrightarrow xBAyCAzBCt &
\textrm{where $\xType{B}=\xType{C}\ne\xType{A}$}.
\end{array}
\end{equation*}
\par
If there is a finite sequence of the homotopy moves H1, H2 and H3, shift
moves and isotopies which transforms one nanoword into another, then
those two nanowords are said to be \emph{homotopic}
\cite{Turaev:Words}.
This relation is an
equivalence relation. Turaev showed
that this idea of homotopy of nanoword 
representations of virtual strings and the usual homotopy of virtual
strings are equivalent \cite{Turaev:KnotsAndWords}. 
That is two nanowords $\alpha$ and $\beta$ are
homotopic if and only if the virtual strings $\Gamma_\alpha$ and
$\Gamma_\beta$ they represent are homotopic.
\par
The homotopy rank of a nanoword $\alpha$, written $\hr(\alpha)$, is the
minimal rank of all nanowords homotopic to $\alpha$ \cite{Turaev:Words}.
This is a homotopy invariant
of $\alpha$.
The homotopy rank of a virtual string $\Gamma$, $\hr(\Gamma)$, is
defined to be the homotopy rank of any nanoword $\alpha$ representing 
$\Gamma$. Clearly this is a homotopy invariant of
$\Gamma$. Geometrically, this invariant is the minimum number of real  
crossings that we need to be able to draw the virtual string as a
virtual string diagram.
\par
In \cite{Turaev:2004}, Turaev defined an invariant for virtual strings
called the $u$-polynomial. We recall the definition here.
\par
We fix a nanoword $\alpha$. Two distinct letters of $\alpha$, $A$ and
$B$, are said to be linked if $A$ and $B$ alternate in $\alpha$ and
unlinked otherwise.
Using this concept the linking number of $A$ and $B$, $\link{A}{B}$ is
defined as follows.
If $A$ and $B$ are unlinked, their linking
number is zero. If $A$ and $B$ are linked, their linking number is
either $1$ or $-1$ depending on the order that $A$ and $B$ appear in
$\alpha$ and on the types of $A$ and $B$:
\begin{equation*}
\link{A}{B}= 
\begin{cases}
0  & \text{$A$ and $B$ are unlinked}, \\
1  & \text{$A$ and $B$ are linked with pattern $\linkpattern{A}{B}{A}{B}$, $\xType{A}=\xType{B}$}, \\
-1 & \text{$A$ and $B$ are linked with pattern $\linkpattern{A}{B}{A}{B}$, $\xType{A}\neq\xType{B}$}, \\
1  & \text{$A$ and $B$ are linked with pattern $\linkpattern{B}{A}{B}{A}$, $\xType{A}\neq\xType{B}$}, \\
-1 & \text{$A$ and $B$ are linked with pattern $\linkpattern{B}{A}{B}{A}$, $\xType{A}=\xType{B}$}.
\end{cases}
\end{equation*}
For completeness, the linking number of any letter $X$ with itself is
defined to be $0$.
\par
Note that the linking number is well-defined under the shift move and that
\begin{equation}\label{eqn:link_skew}
\link{A}{B} = -\link{B}{A}
\end{equation}
for all letters $A$ and $B$ appearing in $\alpha$.
\par
For any letter $X$ in $\alpha$, $\n(X)$ is defined to be the sum of
the linking numbers of $X$ with each of the letters in $\alpha$: 
\begin{equation*}
\n(X) = \sum_{Y \in \alpha}\link{X}{Y}.
\end{equation*}
Note that as $\abs{\link{X}{Y}}$ is less than or equal to $1$ for all
$Y$ and $\link{X}{X}$ is $0$, $\abs{\n(X)}$ is less than $\rank(\alpha)$.
We also note that
\begin{equation}\label{eqn:n_zero_sum}
\sum_{X \in \alpha}\n(X) = \sum_{X \in \alpha}\sum_{Y \in
 \alpha}\link{X}{Y} = 0,
\end{equation}
where the left hand equality is true by definition and the right hand
equality is given by \eqref{eqn:link_skew}.
\par
We remark that $\n(X)$ can be interpreted geometrically by considering a
diagram corresponding to $\alpha$. Note that we can orient $X$ so that
it looks like the crossing on the left of
Figure~\ref{fig:virtualcrossings}. By removing a small neighbourhood of 
the crossing $X$, the curve is split into two segments. We label the
segment starting at the right hand outgoing arc $p$ and the arc
starting at the left hand outgoing arc $q$.
Then $\n(X)$ is the number of times $q$ crosses $p$ from right to
left minus the number of times $q$ crosses $p$ from left to right.
\par
For a positive integer $k$ we define $u_k(\alpha)$ as follows:
\begin{equation*}
u_k(\alpha) = \sharp \lbrace X\in \alpha | \n(X)=k \rbrace - \sharp \lbrace X\in \alpha | \n(X)=-k \rbrace .
\end{equation*}
Here $\sharp$ indicates the number of elements in the set.
Turaev showed that $u_k(\alpha)$ is invariant under homotopy
\cite{Turaev:2004}.  
He combined these invariants into a polynomial called the $u$-polynomial
of $\alpha$ which is defined as  
\begin{equation*}
u_{\alpha}(t) = \sum_{k \geq 1}u_k(\alpha)t^k.
\end{equation*}
\par
As each $u_k(\alpha)$ is invariant under homotopy, it is clear that the
$u$-polynomial is also a homotopy invariant. 
The $u$-polynomial of a virtual string $\Gamma$ is defined to be the
$u$-polynomial of some nanoword $\alpha$ representing $\Gamma$.
\par
We note that for the trivial virtual string
$u_{\trivial}(t)$ is $0$.
We also mention that Theorem~3.4.1 of \cite{Turaev:2004} states that an
integral polynomial 
$u(t)$ can be realized as the $u$-polynomial of a virtual string if and
only if $u(0)=u^{\prime}(1)=0$.
\par
To illustrate the use of the $u$-polynomial we reproduce, in nanoword
terminology, a calculation of the $u$-polynomial of a 2-parameter family
of virtual strings which was orginally made by Turaev in Section~3.3,
Exercise~1 of \cite{Turaev:2004}.
We will use these virtual strings later in this paper.
\begin{ex}\label{ex:alpha_pq}
Consider the virtual string $\Gamma_{p,q}$ for positive integers
$p$ and $q$ represented by the nanoword $\alpha_{p,q}$ given by
\begin{equation*}
X_1X_2\dotso X_pY_1Y_2\dotso Y_qX_p\dotso X_2X_1Y_q\dotso Y_2Y_1
\end{equation*}
where $\xType{X_i}=a$ for all $i$ and $\xType{Y_j}=a$ for all $j$. Then
$\n(X_i)$ is equal to $q$ for all $i$ and $\n(Y_j)$ is equal to $-p$ for
all $j$. So the $u$-polynomial for $\alpha_{p,q}$, and thus
 $\Gamma_{p,q}$, is $pt^q-qt^p$.
When $p$ and $q$ are not equal, the $u$-polynomial is non-zero and so
 $\Gamma_{p,q}$ is non-trivial. In this case, $\Gamma_{p,q}$ is homotopic
 to $\Gamma_{r,s}$ only if $p$ equals $r$ and $q$ equals $s$.
\par
By use of another invariant, the based 
 matrix of a virtual string (which we review in
 Section~\ref{sec:based_matrices}), Turaev showed that the virtual
 strings $\Gamma_{p,p}$ are non-trivial and mutually distinct under
 homotopy for $p$ greater than or equal to $2$ (Section~6.4~(1) of
 \cite{Turaev:2004}).
Using homotopy move H2a on $\alpha_{1,1}$, it is easy to show that
 $\Gamma_{1,1}$ is homotopically trivial (this is also mentioned in
 Section~3.3, Exercise~1 of \cite{Turaev:2004}).
\end{ex}
\par
We now give a definition of the composition of two nanowords
$\alpha$ and $\beta$ which is written $\alpha\beta$.
This operation was originally defined by Turaev in \cite{Turaev:Words}
where he called it multiplication.
\par
If the Gauss words of $\alpha$ and $\beta$ have no letters in common,
the Gauss word of $\alpha\beta$ is the concatenation of the Gauss words
of $\alpha$ and $\beta$. The map from the letters of $\alpha\beta$ to
$\{a,b\}$ is defined by using the map belonging to $\alpha$ for letters
coming from $\alpha$ and the map belonging to $\beta$ for letters
coming from $\beta$.
\par
If the Gauss words of $\alpha$ and $\beta$ do have letters in common, we
can use an isomorphism to transform $\beta$ to a nanoword $\beta^\prime$
that does not have letters in common with $\alpha$. Then the composition
of $\alpha$ and $\beta$ is defined to be the composition of $\alpha$ and
$\beta^\prime$. The following example demonstrates this operation.
\begin{ex}
Let $\alpha$ be the nanoword $\nanoword{ABACDBDC}{abbb}$ and $\beta$ be the nanoword
 $\nanoword{ABACBC}{abb}$. As letters appearing in $\beta$ also appear
 in $\alpha$, we use an isomorphism to get a new nanoword $\nanoword{EFEGFG}{abb}$
 which is isomorphic to $\beta$. We call this new nanoword
 $\beta^\prime$. The composition of $\alpha$ and $\beta$ is then
 the composition $\alpha$ and $\beta^\prime$. We get the nanoword 
$\nanoword{ABACDBDCEFEGFG}{abbbabb}$. 
\end{ex}
\par
In \cite{Turaev:2004} Turaev noted that
\begin{equation}\label{eqn:u-poly_concatenation}
u_{\alpha\beta}(t) = u_{\alpha}(t) + u_{\beta}(t).
\end{equation}
This is because in $\alpha\beta$, letters in $\beta$ do not link any
letters in $\alpha$. Thus for any letter $X$ in $\alpha$, $\n(X)$ in
$\alpha\beta$ is equal to $\n(X)$ in $\alpha$. Similarly, for any letter
$X$ in $\beta$, $\n(X)$ in $\alpha\beta$ is equal to $\n(X)$ in $\beta$.
\par
Composition of nanowords is not well-defined up to homotopy. For example,
we take $\gamma$ to be the trivial nanoword and $\delta$ to be the nanoword
$\nanoword{ABAB}{aa}$. Then $\delta$ is homotopic to $\gamma$ by the
move H2a. On the other hand, the 
composition $\gamma\gamma$ is clearly trivial, yet it can be shown using
primitive based matrices (which we recall later) that $\delta\delta$ is
non-trivial. In fact, $\delta\delta$ represents the virtual string which
underlies Kishino's knot shown in Figure~\ref{fig:kishino}.
\section{Head and tail matrices of a nanoword}\label{sec:th_matrices}
Given a nanoword $\alpha$, we define $\mathcal{A}$ to be the set of letters in
$\alpha$. Then we define two maps,
$t:\mathcal{A}\times\mathcal{A}\rightarrow\{0,1\}$ and
$h:\mathcal{A}\times\mathcal{A}\rightarrow\{0,1\}$.
\par
We set $t(X,X)=h(X,X)=0$ for all $X$ in $\mathcal{A}$. To define $t(X,Y)$ and
$h(X,Y)$, where $X$ and $Y$ are different elements in $\mathcal{A}$, we use the
arrow diagram of the nanoword $\alpha$.
Starting at the letter $X$
at the tail of the arrow joining the two occurences of $X$, we move right along the
nanoword, noting the letters that we move past.
If we reach the end of the nanoword, we return to the start of the
nanoword and continue moving rightwards noting letters.
We keep moving until the letter $X$ at the head of the arrow is found.
If we noted the letter $Y$ at the tail of the arrow joining the two
occurences of $Y$,
$t(X,Y)$ is $1$, otherwise $t(X,Y)$ is $0$. Similarly if we noted the
$Y$ at the head of the arrow, $h(X,Y)$ is $1$, otherwise
$h(X,Y)$ is $0$.
\par
By assigning an order to the letters in $\mathcal{A}$, we can represent the
maps as matrices. We call the matrix representing $t$ the tail matrix
and write it $T(\alpha)$. Similarly the matrix representing $h$ is
called the head matrix and is written $H(\alpha)$. Note that by picking
a different order of the letters in $\mathcal{A}$ we may well get a different
pair of matrices.
\par
\begin{figure}[hbt]
\begin{center}
\includegraphics{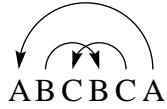}
\caption{Arrow diagram of nanoword $\nanoword{ABCBCA}{bab}$}
\label{fig:abcbca}
\end{center}
\end{figure}
\begin{ex}
Consider the nanoword $\nanoword{ABCBCA}{bab}$, for which the arrow
 diagram is given in Figure~\ref{fig:abcbca}. Ordering the elements
 alphabetically we get this tail matrix:
\begin{equation*}
\begin{pmatrix}
0 & 0 & 0 \\
0 & 0 & 0 \\
1 & 1 & 0
\end{pmatrix}
\end{equation*}
and this head matrix:
\begin{equation*}
\begin{pmatrix}
0 & 0 & 0 \\
0 & 0 & 1 \\
1 & 0 & 0
\end{pmatrix}.
\end{equation*}
\end{ex}
\par
We note that we can calculate the linking number between $X$ and $Y$ in $\alpha$
from the head and tail matrices by
\begin{equation}\label{eqn:linking}
l(X,Y) = t(X,Y) - h(X,Y).
\end{equation}
\par
If $\rank(\alpha)$ is $n$ then $T(\alpha)$ and $H(\alpha)$ are in the set of
$n \times n$ matrices for which all the diagonal entries are $0$ and all
the other entries are either $0$ or $1$.
For a given $n$ we can take any two matrices $T$ and $H$ in this set and
ask whether a nanoword $\alpha$ exists for which $T(\alpha)$ is $T$ and
$H(\alpha)$ is $H$.
\par
A simple restriction comes from the fact that the linking number is
skew-symmetric. Using \eqref{eqn:linking}, any matrices $T$ and
$H$ corresponding to a nanoword $\alpha$ must satisfy
\begin{equation}\label{eqn:restriction}
T-H = -\transpose{(T-H)},
\end{equation}
where $t$ means the matrix transpose operation.
Unfortunately this is not a sufficient condition.
For example, consider the pair of matrices $T$ and $H$ given by
\begin{equation*}
\begin{pmatrix}
0 & 0 & 0 \\
1 & 0 & 0 \\
0 & 1 & 0
\end{pmatrix}
\text {and}
\begin{pmatrix}
0 & 1 & 0 \\
0 & 0 & 0 \\
0 & 1 & 0
\end{pmatrix}.
\end{equation*}
These matrices satisfy \eqref{eqn:restriction} but a simple
combinatorial check shows that there is no 3-letter nanoword $\alpha$ to which
they correspond.
\section{Based matrices of virtual strings}\label{sec:based_matrices}
In \cite{Turaev:2004}, Turaev introduced the concept of a based matrix
and described how to associate a based matrix with a virtual string. We
briefly recall the definitions here.
\par
Let $G$ be a finite set with a special element $s$. For some abelian
group $H$ let $b$ be a map from $G\times G$ to $H$ satisfying
$b(g,h) = -b(h,g)$
for all $g$ and $h$ in $G$ (in other words, $b$ is skew-symmetric). Then
the triple $(G,s,b)$ is a based matrix over $H$.
For the rest of this paper we will take $H$ to be $\Z$.
\par
We can associate a based matrix with a nanoword $\alpha$ as
follows. First we take $G$ to be the set of letters in $\alpha$ union
the special element $s$.
\par
We then consider a diagram corresponding to $\alpha$ embedded in some
surface $S$.
To each element $g$ of $G$ we associate a closed loop in the surface $S$
which we label $g_c$.
For the special element $s$ we define $s_c$ to be the whole curve.
Any other element in $G$ corresponds to a crossing in the diagram. For
any crossing $X$ we define a closed loop $X_c$ as follows. First we
orient the crossing so that it looks like the crossing on the left of
Figure~\ref{fig:virtualcrossings}. Then starting from $X$ we leave the
crossing on the outgoing right hand arc and follow the curve until we
get back to $X$ for the first time. We define $X_c$ to be a loop
parallel to the loop we have just traced. Figure~\ref{fig:loop} shows an
example.
\par
We define $b(g,h)$ to be the homological intersection number of the loop
$g_c$ with the loop $h_c$. This is just the number of times that
$h_c$ crosses $g_c$ from right to left minus the number of times that
$h_c$ crosses $g_c$ from left to right.
\begin{figure}[hbt]
\begin{center}
\includegraphics{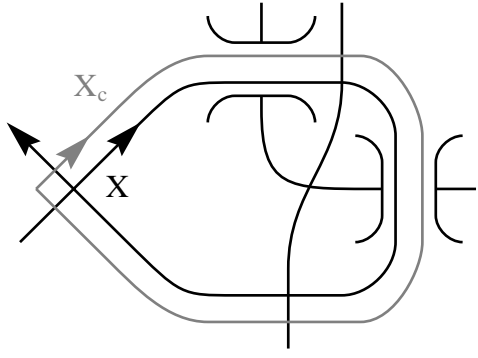}
\caption{Example of defining a loop at a crossing}
\label{fig:loop}
\end{center}
\end{figure}
\par
We then define the based matrix associated with $\alpha$ to be the
triple $(G,s,b)$. We write this $M(\alpha)$.
\par
By assigning an order to the elements in $G$ it is possible to write $b$
as a matrix. By convention, the special element $s$ always comes first
in such an ordering. Thus the first row of the matrix has elements of
the form $b(s,x)$ and the first column has elements of the form $b(x,s)$
for each $x$ in $G$. The resulting matrix is skew-symmetric.
\par
Let $(G_1,s_1,b_1)$ and $(G_2,s_2,b_2)$ be two based matrices.
If there is a bijection $f$ from $G_1$ to $G_2$ such that $f(s_1)$ equals
$s_2$ and for all $g$ and $h$ in $G$, $b_2(f(x),f(y))$ is equal to
$b_1(g,h)$, then the two based matrices are said to be isomorphic
\cite{Turaev:2004}.
Informally, two based matrices are isomorphic if we
can pick orderings of the elements of the two sets $G_1$ and $G_2$ such
that $s_1$ and $s_2$ are first in their respective sets and the
corresponding skew-symmetric matrices are the same.
\par
We can calculate $b$ directly from $\alpha$.
It is clear that $b(s,s)$ is $0$.
In \cite{Turaev:2004}, Turaev showed that $b(g,s)$ is equal to $\n(g)$
for all $g$ in $G-\lbrace s \rbrace$.
In other words, $\n(X)$ is the homological intersection number of $X_c$
and the whole curve in the diagram.
For $g$ and $h$ in $G-\lbrace s\rbrace$ we can calculate $b(g,h)$ by
\begin{equation}\label{eqn:matrix_elements}
b(g,h) = t(g,h) - h(g,h) + 
\sum_{k \in G-\lbrace s\rbrace}\bigl( t(g,k)h(h,k) - h(g,k)t(h,k) \bigr)
\end{equation}
where $t$ and $h$ were defined in Section~\ref{sec:th_matrices} (this
equation is derived from Lemma~4.2.1 in \cite{Turaev:2004}).
We can thus write $b$ as a matrix in the form
\begin{equation*}
\begin{pmatrix}
0& -\transpose{\cvector{n}} \\
\cvector{n}& B
\end{pmatrix}
\end{equation*}
where $B$ is the submatrix of $b$ corresponding to the elements in
$G-\lbrace s \rbrace$
and $\cvector{n}$ is the column vector consisting of $\n(g)$ for each
$g$ (where the order of elements in the vector matches the order of the
elements in the matrix $B$).
By \eqref{eqn:matrix_elements} we can calculate $B$ directly from
$T(\alpha)$ and $H(\alpha)$ using this formula 
\begin{equation*}
B = T - H + T \transpose{H} - H \transpose{T}
\end{equation*}
where we have written $T$ for $T(\alpha)$ and $H$ for $H(\alpha)$.
\par
Turaev made the following definitions \cite{Turaev:2004}.
An annihilating element of a
based matrix $(G,s,b)$ is an element $g$ in $G-\lbrace s\rbrace$ for
which $b(g,h)=0$ for all $h$ in $G$.
A core element is an element $g$ in
$G-\lbrace s\rbrace$ for which $b(g,h) = b(s,h)$ for all $h$ in $G$.
Two elements $g$ and $h$ in $G-\lbrace s\rbrace$ are complementary
elements if $b(g,k) + b(h,k) = b(s,k)$ for all $k$ in $G$.
A based matrix is called primitive if it has no annihilating elements,
core elements or complementary elements.
\par
Turaev defined three reducing operations on based matrices
\cite{Turaev:2004}.
The first 
removes an annihilating element, the second removes a core element and
the third removes a complementary pair. Each move transforms a based
matrix $(G,s,b)$ to $(G^\prime,s,b^\prime)$, where $G^\prime$ is
derived from $G$
by removing the element(s) involved in the operation and $b^\prime$ is
$b$ restricted to $G^\prime$. 
\par
A primitive based matrix does
not admit any of these reducing operations. Clearly we can apply a sequence
of these operations to a non-primitive based matrix until we derive a
primitive based matrix. Turaev showed that up to
isomorphism, the resulting primitive based matrix is the same,
irrespective of which elements we remove and the order in which we
remove them \cite{Turaev:2004}. 
\par
We can apply these reducing operations to the based matrix $M(\alpha)$
associated with the nanoword $\alpha$ to get a primitive based matrix. We
call it $P(\alpha)$. Turaev showed that up to isomorphism $P(\alpha)$ is
a homotopy invariant of $\alpha$ \cite{Turaev:2004}. Thus we can define
the primitive based 
matrix $P(\Gamma)$ of a virtual string $\Gamma$ to be $P(\alpha)$ for
any nanoword representing $\Gamma$. In particular this means that we can use
properties of $P(\Gamma)$ that are invariant under isomorphism of based
matrices as invariants of virtual strings.
Turaev gave some suggestions for such invariants in \cite{Turaev:2004}.
In \cite{Gibson:tabulating-vs} we define a canonical representation of a
based matrix which can be used as a complete invariant of based
matrices up to isomorphism.
\par
A simple invariant of primitive based matrices that we will use in this
paper is the number of elements in the set in $P(\Gamma)$. As the
special element $s$ cannot be removed by any of the moves defined on
based matrices, a based matrix always has at least one element. Thus
$\rho(\Gamma)$ is defined to be the number of elements in $P(\Gamma)$
minus one. This invariant was defined in \cite{Turaev:2004}.
\par
Turaev also noted that the $u$-polynomial of a nanoword
$\alpha$ can be calculated from $M(\alpha)$ or $P(\alpha)$ (because, as
we have mentioned, $b(g,s)$ is equal to $\n(s)$ for all $g$ in
$G-\lbrace s \rbrace$) \cite{Turaev:2004}.
This means that for nanowords $\alpha$ and $\beta$, if $P(\alpha)$ is
isomorphic to $P(\beta)$ then $u_{\alpha}(t)$ is equal to
$u_{\beta}(t)$. However the converse is not necessarily true and the
primitive based matrix invariant is stronger than the $u$-polynomial
\cite{Turaev:2004}.
\section{Coverings}\label{sec:coverings}
Turaev defined an operation on virtual strings called a covering
\cite{Turaev:2004}.
He also defined coverings for nanowords in \cite{Turaev:Words}.
Given a nanoword $\alpha$ and a non-negative integer $r$, the
\textcover{r} of $\alpha$ is the nanoword derived from $\alpha$ by
removing any letter $X$ in $\alpha$ for which $\n(X)$ is not divisible
by $r$.
The \textcover{r} of $\alpha$ is written $\cover{\alpha}{r}$.
\par
For a virtual string $\Gamma$ we pick a nanoword $\alpha$ that
represents it and then define $\cover{\Gamma}{r}$ to be the virtual
string realized by $\cover{\alpha}{r}$.
In \cite{Turaev:2004}, Turaev showed that $\cover{\Gamma}{r}$ is not
dependent on the nanoword $\alpha$ that we picked. That is, if $\alpha_1$
and $\alpha_2$ are homotopic nanowords representing the virtual string $\Gamma$,
$\cover{\alpha_1}{r}$ and $\cover{\alpha_2}{r}$ are also homotopic. This
means that $\cover{\Gamma}{r}$ is
an invariant of the virtual string $\Gamma$ and invariants of coverings
of virtual strings can
be used to distinguish the virtual strings themselves. We call $\cover{\Gamma}{r}$ the
\textcover{r} of $\Gamma$.
\par
Note that we have
extended Turaev's definition to include the case where $r$ is
$0$. Turaev's invariance result is true in this case too.
\par
A covering is thus a map from the set of virtual strings to itself. It
is interesting to ask such questions as whether the map is injective or
surjective and whether there exist any fixed points. We can also ask
what happens when we repeatedly apply the map to its own output. Are
there any periodic points?
\par
Note that for any virtual string $\Gamma$, the \textcover{1} of $\Gamma$
is always $\Gamma$ and so these questions are easily answered when $r$
is $1$. We also note that for all $r$, $\cover{\trivial}{r}$ is
$\trivial$.
\par
\begin{ex}\label{ex:simple_covering}
Consider the virtual string $\Gamma$ represented by the nanoword $\alpha$ given by
$\nanoword{ABCACB}{aaa}$. We have $\n(A)=2$ and $\n(B)=\n(C)=-1$. 
So we have $u_\Gamma(t)=t^2-2t$ and $\Gamma$ is not trivial.
When $r$ is 2, $\cover{\alpha}{2}$ is $\nanoword{AA}{a}$ which is homotopically
trivial by the first homotopy move and so $\cover{\Gamma}{2}$ is $\trivial$.
When $r$ is $0$ or greater than $2$, $\cover{\alpha}{r}$ is $\trivial$ and so
$\cover{\Gamma}{r}$ is also $\trivial$. Thus for all $r$ not equal to $1$,
$\cover{\Gamma}{r}$ is trivial and equal to $\trivial$.
\end{ex}
\par
This example shows that the covering map for $r$ is not injective unless
$r$ is equal to $1$.
\par
\begin{thm}\label{thm:surjectivity}
For any non-negative integer $r$, the covering map corresponding to $r$ is 
surjective. When $r$ is not $1$, for any given virtual string $\Gamma$ there
are an infinite number of virtual strings which map to $\Gamma$ under the covering map.
\end{thm}
\begin{proof}
We have already observed that when $r$ is $1$ the first claim is
 true. We consider the case when $r$ is not $1$. We show that given a
 virtual string $\Gamma$ represented by a nanoword $\alpha$, we can construct a new
 nanoword $\beta$ for which $\cover{\beta}{r}$ is $\alpha$. Then $\beta$
 represents a virtual 
 string which maps to $\Gamma$ under the covering map corresponding to $r$.
\par
To construct $\beta$ we first make a copy of $\alpha$. Then
for each letter $X$ in $\alpha$, we consider the value of $\n(X)$.
\par
If $\n(X)$ is zero then we make no changes relating to $X$.
\par
If $\n(X)$
 is non-zero we add letters to $\beta$ in the following way
\begin{equation*}
 xXyXz \longrightarrow xXyA_1A_2\dotso A_kXA_k\dotso A_2A_1z
\end{equation*}
where $k$ equals $\abs{\n(X)}$ and, for all $i$, $\xType{A_i}$ is set to
 $\xType{X}$ if $\n(X)$ is negative and the opposite type to $\xType{X}$
 if $\n(X)$ is positive.  Note that in the new nanoword $\n(A_i)$ is
 $\pm 1$ for all $i$
 and $\n(X)$ is $0$. Also note that for any other letter $Y$ in the old
 nanoword, $\n(Y)$ is unchanged as we go from the old nanoword to the new nanoword.
\par
Once we have considered every letter in $\alpha$ and made the
 appropriate additions, we call the resultant
nanoword $\beta$. The nanoword consists of letters $X_i$ originally in $\alpha$,
and letters $A_j$ which we added. By construction, in $\beta$, $\n(X_i)$
is $0$ for all $i$ and $\n(A_j)$ is $\pm 1$ for all $j$. Thus when we
take the \textcover{r} ($r$ not $1$) of $\beta$, we remove all the
 letters $A_j$
and keep all the letters $X_i$. Since the order and the types of the
 letters $X_i$ were
not changed during our construction of $\beta$, the result is
 $\alpha$. Thus $\cover{\beta}{r}$ is $\alpha$ and the first claim of
 the theorem is proved.
\par
To prove the second claim we use the fact that for a letter $X$ in a
 nanoword $\gamma$ or $\delta$, $\n(X)$ remains unchanged in the composition
 $\gamma\delta$. This implies that
\begin{equation*}
\cover{(\gamma\delta)}{r} = \cover{\gamma}{r}\cover{\delta}{r}.
\end{equation*}
\par
It is simple to
calculate that the $u$-polynomial for the virtual string $\beta$
 constructed above is $0$.
\par
Consider the nanowords $\alpha_{p,1}$ and $\alpha_{1,p}$ in Example
 \ref{ex:alpha_pq}. These have 
$u$-polynomials $pt-t^p$ and $t^p-pt$ respectively. If we take the
 \textcover{r} of either nanoword ($r$ not $1$), we either get the
 trivial nanoword (if $r$ does not divide $p$) or a nanoword
 isomorphic to $\nanoword{AA}{a}$ (if $r$ does divide $p$). In the
 latter case, the letter $A$ can then be removed by the move H1. The
 result in either case is the trivial nanoword.
\par
Therefore, if we take the composition of $\beta$ with such
nanowords we can construct new nanowords such that the \textcover{r} is
 still $\alpha$. However by \eqref{eqn:u-poly_concatenation}
the $u$-polynomial will be non-zero. In particular we can
construct an infinite family of nanowords $\beta\alpha_{1,p}$ for which
$\cover{(\beta\alpha_{1,p})}{r}$ is $\alpha$ and, by
 \eqref{eqn:u-poly_concatenation}, the $u$-polynomial is
$t^p-pt$ (for $p$ greater than $1$). Thus the family of virtual strings
 that they represent are all mutually homotopically distinct.
\end{proof}
Note that in fact, by concatenating $\beta$ with multiple copies of
$\alpha_{p,1}$ and $\alpha_{1,p}$, possibly with different values of
$p$, it is possible to construct a nanoword $\gamma$ with any $u$-polynomial 
satisfying $u(0)=u^{\prime}(1)=0$
such that $\cover{\gamma}{r}$ is $\alpha$. Thus we have the
corollary:
\begin{cor}
For any virtual string $\Gamma$, any $u$-polynomial $u(t)$ satisfying
 $u(0)=u^{\prime}(1)=0$ and any
 non-negative integer $r$, $r$ not $1$, there is a virtual string $\Lambda$ such that
 $u_\Lambda(t)=u(t)$ and $\cover{\Lambda}{r}$ is $\Gamma$. 
\end{cor}
\par
We can use coverings to define some numeric invariants of virtual
strings. The following proposition suggests one such invariant.
\par
\begin{prop}
For a virtual string $\Gamma$, there
exists an integer $m$ such that for all $n$ greater than or equal to $m$,
$\cover{\Gamma}{n}$ is $\cover{\Gamma}{0}$.
\end{prop}
\begin{proof}
Consider $\alpha$ a nanoword with finite rank which represents $\Gamma$. Then for any
 letter $X$ in $\alpha$, we have already observed that $\abs{\n(X)}$ is less
 than $\rank(\alpha)$. If we set $m$ to be 
 $\rank(\alpha)$ then for all $n$ greater than or equal to $m$, only the
 letters $X$ in $\alpha$ with $\n(X)$ equal to $0$ will appear in the
 \textcover{n}. Thus $\cover{\Gamma}{n}$ is $\cover{\Gamma}{0}$.
\end{proof}
Thus we can define $\m(\Gamma)$ to be the minimal integer $m$ such
that for all $n$ greater than or equal to $m$, $\cover{\Gamma}{n}$ is
$\cover{\Gamma}{0}$. The proposition shows that $\m(\Gamma)$ is
always defined. Of course, the minimal such $m$ may be less than
$\rank(\alpha)$. In Example \ref{ex:simple_covering} the rank of the
initial nanoword was $3$ but all the coverings except for the
\textcover{1} were trivial.
\par
The proof of the proposition shows that $\m(\Gamma)$ is less than or
equal to $\hr(\Gamma)$. In fact, the proof shows that if $m$ is
greater than the largest $\abs{\n(X)}$ for an $X$ in an $\alpha$
representing $\Gamma$, then $\cover{\Gamma}{m}$ is $\cover{\Gamma}{0}$.
Therefore,
\begin{equation*}
\m(\Gamma) \leq \max \bigl\{ \abs{\n(X)} \ \bigm| X\in \alpha \bigr\} +1
\end{equation*}
where $\alpha$ represents $\Gamma$ and $\rank(\alpha)$ is equal to
$\hr(\Gamma)$. 
If all the letters $X$ in $\alpha$ have $\n(X)$ equal to $0$ then
$\m(\Gamma)$ is $0$. In particular $\m(\trivial)$ is $0$.
\par
\begin{prop}
For a virtual string $\Gamma$ and a
 non-negative integer $r$ such that $\cover{\Gamma}{r}$ is not equal to $\Gamma$,
\begin{equation*}
 \hr(\cover{\Gamma}{r}) \leq \hr(\Gamma)-2.
\end{equation*}
\end{prop}
\begin{proof}
For a given virtual string $\Gamma$, we can
find a nanoword $\alpha$ which represents $\Gamma$ and has minimal rank.
So $\rank(\alpha)$ is equal to $\hr(\Gamma)$. When we take the
 \textcover{r}, we get a nanoword $\cover{\alpha}{r}$ representing
 $\cover{\Gamma}{r}$. By assumption, 
 $\cover{\alpha}{r}$ does not equal $\alpha$, so we must have
deleted some letters from $\alpha$ to get the \textcover{r}.
\par
Assume that we deleted just a single letter $X$ from $\alpha$. 
We deal with two cases, $r$ is greater than or equal to $2$ and $r$ is
 equal to $0$. It is not possible that $r$ is $1$ because of the assumption that
 $\cover{\Gamma}{r}$ is not equal to $\Gamma$.
\par
When $r$ is greater than or equal to $2$,
$\n(X)$ must be
 equal to $pr+q$, $(0<q<r)$, for some integers $p$ and $q$.
All other letters $Y$ in $\alpha$ must have $\n(Y)$ equal to $kr$
 for some integer $k$ dependent on $Y$. We now calculate the sum of
 $\n(Z)$ for all letters $Z$ in $\alpha$.
\begin{equation*}
\sum_{Z \in \alpha}\n(Z) = \n(X) + \sum_{Y \neq X}\n(Y) \equiv q \pmod{r}.
\end{equation*}
However, by \eqref{eqn:n_zero_sum} the left hand side is $0$ and
 we have a contradiction.
\par
When $r$ is equal to $0$, $\n(X)$ must be equal to some non-zero
 integer $q$. All other letters $Y$ in $\alpha$ must have $\n(Y)$ equal
 to $0$. Then we have
\begin{equation*}
0 = \sum_{Z \in \alpha}\n(Z) = \n(X) + \sum_{Y \neq X}\n(Y) = q \neq 0
\end{equation*}
which is a contradiction.
\par
Thus, in either case, we must have deleted at least two letters. We have
\begin{equation*}
\hr(\cover{\Gamma}{r}) \leq \rank(\cover{\alpha}{r}) \leq
 \rank(\alpha)-2 = \hr(\Gamma)-2.
\end{equation*}
\end{proof}
\par
In particular the homotopy rank of the \textcover{r} can never be bigger
than the homotopy rank of the original virtual string.
\par
Fixing $r$, we can define a sequence of virtual strings for any virtual
string $\Gamma$ as follows. Set $\Gamma_0$ to be
$\Gamma$. Then define $\Gamma_i$ to be 
$\cover{(\Gamma_{i-1})}{r}$ for $i\geq 1$.
As the number of crossings in $\Gamma$ is finite and cannot increase when we
take the cover, there exists an $n$ such
that $\Gamma_{n+1}$ is equal to $\Gamma_{n}$. We can thus make the following definitions:
\begin{equation*}
\height_r(\Gamma):=\min \lbrace n|\Gamma_{n+1} = \Gamma_{n} \rbrace
\end{equation*}
and
\begin{equation*}
\base_r(\Gamma):=\Gamma_{\height_r(\Gamma)}.
\end{equation*}
It is clear that these are invariants of $\Gamma$.
\par
For $r$ not $1$, we can represent the action of the \textcover{r} map on
the set of 
virtual strings as a directed graph. The vertices of the graph
represent the virtual strings. For each virtual string $\Gamma$ we draw an
oriented edge from the vertex which represents it to the vertex which
represents $\cover{\Gamma}{r}$. By the above discussion it is clear that each
connected component of the graph will take the form of a tree with a
loop at the root point. Every vertex has an infinite number of
incoming edges and a single outgoing edge. Figure~\ref{fig:graph}
depicts a small part of the graph near the base of a single component.
\begin{figure}[hbt]
\begin{center}
\includegraphics{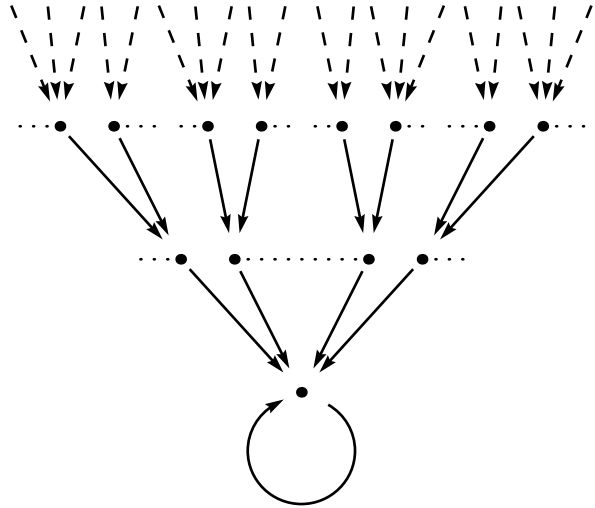}
\caption{Part of a single component of the graph of a covering map}
\label{fig:graph}
\end{center}
\end{figure}
\section{Fixed points under coverings}\label{sec:fixed}
For a fixed $r$, we can define the set of fixed points under the
\textcover{r} map by
\begin{equation*}
\baseSet{r}:= \lbrace \Gamma \in \setStrings | \cover{\Gamma}{r} = \Gamma \rbrace.
\end{equation*}
\par
Note that we could equivalently define the set by
\begin{equation*}
\baseSet{r}= \lbrace \base_r(\Gamma)| \Gamma \in \setStrings \rbrace,
\end{equation*}
or indeed by
\begin{equation*}
\baseSet{r}= \lbrace \Gamma \in \setStrings | \height_r(\Gamma) = 0 \rbrace.
\end{equation*}
\par
When $r$ is $1$, $\baseSet{1}$ is $\setStrings$. For other $r$ this is
not the case. Indeed, we showed that for any given virtual string
$\Gamma$ there are infinitely many virtual strings which are different
from $\Gamma$ but for which the \textcover{r} is $\Gamma$. Those virtual
strings are not in the set $\baseSet{r}$. On the other hand, we have
already noted that the trivial virtual string is in $\baseSet{r}$ for
all $r$. We now give a method for constructing more examples of virtual
strings in $\baseSet{r}$ for $r$ greater than $1$.
\par
Given a nanoword $\alpha$ and an integer $r$ greater than $1$ we define
a new nanoword from $\alpha$ in the following way. For each letter $A$
in $\alpha$ we replace the first occurence of $A$ with $r$ letters
$A_1A_2\dotso A_r$ and the second occurence of $A$ with $r$ letters
$A_r\dotso A_2A_1$ where the letters $A_i$ have the same type as $A$ for
all $i$. We call the resultant nanoword $r\cdot\alpha$. The construction
$r\cdot\alpha$ appears in Section~3.7, Exercise~2 of
\cite{Turaev:2004}.
\par
As an example, if $\alpha$ is the nanoword $\nanoword{ABACBC}{aab}$,
then $2\cdot\alpha$ is the nanoword
$A_1A_2B_1B_2A_2A_1C_1C_2B_2B_1C_2C_1$ where letters $C_1$ and $C_2$ are
of type $b$ and the other letters are of type $a$. 
\par
We note that this operation is not well-defined
for virtual strings. If $\alpha$ and $\beta$ are homotopic nanowords, it
is not necessarily true that $r\cdot\alpha$ and $r\cdot\beta$ are
homotopic. An example suffices to show this. Take $\alpha$ to be the
nanoword $\nanoword{ABCBDCAD}{aabb}$ and $\beta$ to be the nanoword
$\nanoword{BACDBCDA}{aabb}$. Then $\alpha$ and $\beta$ are homotopic (in
fact they are related by an H3b move involving $A$, $B$ and $D$) but, by
using based matrices, we can show that $2\cdot\alpha$ and $2\cdot\beta$
are not homotopic.
\par
We now consider the behaviour of $r\cdot\alpha$ under \textcover{r}.
For any letter $A_i$ in $r\cdot\alpha$, $\n(A_i)$ is equal to
$r\n(A)$. Thus $r\cdot\alpha$ is fixed under the \textcover{r} map. So
the virtual string represented by $r\cdot\alpha$ is in $\baseSet{r}$.
\par
We now note that if $\Gamma$ is in $\baseSet{r}$ then any nanoword
$\alpha$ representing $\Gamma$ which satisfies
$\rank(\alpha)=\hr(\alpha)$ must also satisfy 
$\cover{\alpha}{r}=\alpha$. 
Then $\alpha$ consists 
only of letters $X$ with $\n(X)$ equal to $kr$ for some integer $k$
dependent on $X$. Using this fact we can easily find
relationships between the sets of fixed points. We have
\begin{equation}\label{eqn:subset_zero}
\baseSet{0} \subset \baseSet{r}
\end{equation}
for all $r$ and
\begin{equation}\label{eqn:subset_multiple}
\baseSet{kr} \subset \baseSet{r}
\end{equation}
for all $k\geq 2$ and for all $r \geq 2$.
\begin{prop}
For distinct natural numbers $p$ and $q$, write $l$ for the lowest common
 multiplier of $p$ and $q$, and write $g$ for the greatest common
 divisor of $p$ and $q$. Then
\begin{equation}\label{eqn:base_lcm}
\baseSet{p} \cap \baseSet{q} = \baseSet{l}
\end{equation}
and
\begin{equation}\label{eqn:base_gcd}
\baseSet{p} \cap \baseSet{q} \subset \baseSet{g}.
\end{equation}
\end{prop}
\begin{proof}
If $\Gamma$ is in $\baseSet{p} \cap \baseSet{q}$, then by the remark above
 there exists a nanoword $\alpha$ which represents $\Gamma$ for which
$\cover{\alpha}{p}=\alpha$ and $\cover{\alpha}{q}=\alpha$.
Then $\alpha$ consists only of letters $X$ with $\n(X)=0$ or $\n(X)$
 divisible both by $p$ and by $q$. In the latter case, this means that
 $l$ also divides $\n(X)$. Thus $\Gamma$ is in $\baseSet{l}$ and
\begin{equation*}
\baseSet{p} \cap \baseSet{q} \subseteq \baseSet{l}.
\end{equation*}
As $l$ is a multiple of both $p$ and $q$, $\baseSet{l}$ is a subset of
 both $\baseSet{p}$ and $\baseSet{q}$ by
 \eqref{eqn:subset_multiple}.
This proves \eqref{eqn:base_lcm}.
\par
As $l$ is a multiple of $g$, we can use
 \eqref{eqn:subset_multiple} and \eqref{eqn:base_lcm} to show
 \eqref{eqn:base_gcd}.
\end{proof}
\begin{prop}
We have
\begin{equation*}
\bigcap_{i=1}^\infty \baseSet{i} = \baseSet{0}.
\end{equation*}
\end{prop}
\begin{proof}
From \eqref{eqn:subset_zero} we can see that
\begin{equation*}
\bigcap_{i=1}^\infty \baseSet{i} \supseteq \baseSet{0}.
\end{equation*}
\par
Now assume that equality does not hold. That is, we assume there exists a
 virtual string $\Gamma$ in $\setStrings$ which is in the intersection but not in
 $\baseSet{0}$. Then for all $r$ greater $0$, $\cover{\Gamma}{r}$ is
 $\Gamma$. However, as $\Gamma$ is in $\setStrings$, $\Gamma$ is
 represented by a nanoword $\alpha$ which has finite rank $n$.
\par
Consider $\cover{\alpha}{n}$. This must be $\alpha$ by our assumption that
 $\Gamma$ is in the intersection. This means that every letter $X$ in
 $\alpha$ has $\n(X)$ equal to $nk$ for some integer $k$. Now as $\alpha$ has
 rank $n$, we know that $\abs{\n(X)}$ must be less than $n$ for all
 $X$. Thus $\n(X)$ must be $0$ for all $X$. However, this means that
 $\cover{\alpha}{0}$ is $\alpha$ and thus that $\Gamma$ is in
 $\baseSet{0}$ which contradicts our initial assumption. Thus equality
 holds and the proof is complete.
\end{proof}
\begin{thm}
For $r$ greater than $1$, $\baseSet{r}$ is infinite in size.
\end{thm}
\begin{proof}
Recall the virtual string $\Gamma_{p,q}$ from Example \ref{ex:alpha_pq}.
Considering the cover $\cover{(\Gamma_{p,q})}{r}$, it is clear
that this is $\Gamma_{p,q}$ if $r$ divides both $p$ and $q$, and trivial
otherwise. Thus $\Gamma_{jr,kr}$ for $j,k\geq 1$ is a 2-parameter family
 of virtual strings which are fixed under the \textcover{r} and so are
 all in $\baseSet{r}$. In Example \ref{ex:alpha_pq} we saw that these
 virtual strings are mutually distinct. Thus $\baseSet{r}$ is infinite
 in size.
\end{proof}
Now consider the two families $\Gamma_{kr,r}$ and
$\Gamma_{r,kr}$. Elements in these families have the property that if we
take the \textcover{p} for $p$ greater than $r$ we get the trivial
virtual string. We also note that the \textcover{0} of these virtual
strings are trivial. Thus the set 
\begin{equation*}
\baseSet{r} - \baseSet{0} - \bigcup_{i=r+1}^{\infty}\baseSet{i}
\end{equation*}
is infinite in size.
\par
We now consider the set $\baseSet{0}$. We know that the trivial virtual
string is in this set. We now show that this set has other members.
\par 
Consider the family of virtual strings $\Gamma_n$ for $n \geq 3$
represented by the nanowords $\alpha_n$ given by
\begin{equation*}
X_0X_{n-1}X_1X_0X_2X_1 \dotso X_{i}X_{i-1}X_{i+1}X_{i} \dotso X_{n-2}X_{n-3}X_{n-1}X_{n-2}
\end{equation*}
where $\xType{X_i}$ is $a$ for $0 \leq i \leq n-2$ and $\xType{X_{n-1}}$ is
 $b$.
Then $\n(X_i)$ is $0$ for all $i$ and so $\cover{(\alpha_n)}{0}$ is $\alpha_n$ for
 all $n$. Thus $\alpha_n$ is in $\baseSet{0}$. By appropriate shift and
 homotopy moves it is possible to show that when $n$ is $3$, $4$ or $6$,
 $\alpha_n$ is homotopically trivial. For the remaining $n$ we have the
 following lemma.
\begin{lem}
For $n=5$ or $n \geq 7$ the nanowords $\alpha_n$ are non-trivial and mutually
 distinct under homotopy. 
\end{lem}
\begin{proof}
In the following calculations we consider the subscripts of
 the letters $X_i$ to be in $\cyclic{n}$.
\par
We have
\begin{equation}\label{eqn:baseSet0Tail}
t(X_i,X_j) =
\begin{cases}
1 & \text{if $j=i+1$} \\
0 & \text{otherwise}
\end{cases}
\end{equation}
and
\begin{equation}\label{eqn:baseSet0Head}
h(X_i,X_j) =
\begin{cases}
1 & \text{if $j=i-1$} \\
0 & \text{otherwise.}
\end{cases}
\end{equation}
Now
\begin{align*}
b(X_i,X_j) = & t(X_i,X_j) - h(X_i,X_j) + \\
& \sum_X \bigl( t(X_i,X)h(X_j,X) - h(X_i,X)t(X_j,X) \bigr)
\end{align*}
where the sum is taken over all the letters $X_k$. However, by
 \eqref{eqn:baseSet0Head} and \eqref{eqn:baseSet0Tail} most of the elements
 in the sum will be zero and so we only need to consider the cases when
 $X$ is $X_{i-1}$, $X_{i+1}$, $X_{j-1}$ or $X_{j+1}$. In the
 calculations below we write $s(X_i,X_j)$ for the sum
\begin{equation*}
\sum_X \bigl( t(X_i,X)h(X_j,X) - h(X_i,X)t(X_j,X) \bigr).
\end{equation*}
We also assume
 that $n \geq 5$.
\par
Of course, if $i$ equals $j$, $b(X_i,X_j)$ is $0$.
We consider the following cases: $i=j-2$; $i=j-1$; $i=j+1$; $i=j+2$;
 $\abs{i-j} \geq 3$.
\par
Case $i=j-2$: as $X_{j-1} = X_{i+1}$
the sum is over $3$ elements ($X_{i-1}$, $X_{i+1}$ and $X_{i+3}$).
\begin{align*}
s(X_i,X_{i+2}) = 
 & t(X_i,X_{i-1})h(X_{i+2},X_{i-1}) - h(X_i,X_{i-1})t(X_{i+2},X_{i-1}) + \\
 & t(X_i,X_{i+1})h(X_{i+2},X_{i+1}) - h(X_i,X_{i+1})t(X_{i+2},X_{i+1}) + \\
 & t(X_i,X_{i+3})h(X_{i+2},X_{i+3}) - h(X_i,X_{i+3})t(X_{i+2},X_{i+3}) \\
= & 0 - t(X_{i+2},X_{i-1}) + 1 - 0 + 0 - h(X_i,X_{i+3}) \\
= & 1
\end{align*}
as $t(X_j,X_{j-3})$ and $h(X_i,X_{i+3})$ are both
 $0$ when $n \geq 5$. Thus 
\begin{align*}
b(X_i,X_{i+2}) = & t(X_i,X_{i+2}) - h(X_i,X_{i+2}) + s(X_i,X_{i+2}) \\
= & 1.
\end{align*}
\par
Case $i=j-1$: we have $X_{j-1} = X_i$
and $X_{j+1} = X_{i+2}$. Thus 
 the sum is over $4$ elements ($X_{i-1}$, $X_i$, $X_{i+1}$ and
 $X_{i+2}$).
\begin{align*}
s(X_i,X_{i+1}) =
 & t(X_i,X_{i-1})h(X_{i+1},X_{i-1}) - h(X_i,X_{i-1})t(X_{i+1},X_{i-1}) + \\
 & t(X_i,X_i)h(X_{i+1},X_i) - h(X_i,X_i)t(X_{i+1},X_i) + \\
 & t(X_i,X_{i+1})h(X_{i+1},X_{i+1}) - h(X_i,X_{i+1})t(X_{i+1},X_{i+1}) + \\
 & t(X_i,X_{i+2})h(X_{i+1},X_{i+2}) - h(X_i,X_{i+2})t(X_{i+1},X_{i+2}) \\
= & 0.
\end{align*}
Thus 
\begin{align*}
b(X_i,X_{i+1}) = & t(X_i,X_{i+1}) - h(X_i,X_{i+1}) + s(X_i,X_{i+1}) \\
= & 1.
\end{align*}
\par
Case $i=j+1$: by skew-symmetry and the case $i=j-1$ we have
\begin{equation*}
b(X_i,X_{i-1}) = -b(X_{i-1},X_i) = -1.
\end{equation*}
\par
Case $i=j+2$: by skew-symmetry and the case $i=j-2$ we have
\begin{equation*}
b(X_i,X_{i-2}) = -b(X_{i-2},X_i) = -1.
\end{equation*}
\par
Case $\abs{i-j} \geq 3$: in this case we have
\begin{equation*}
t(X_j,X_{i-1})=t(X_j,X_{i+1})=t(X_i,X_{j-1})=t(X_i,X_{j+1})=0
\end{equation*}
and
\begin{equation*}
h(X_j,X_{i-1})=h(X_j,X_{i+1})=h(X_i,X_{j-1})=h(X_i,X_{j+1})=0.
\end{equation*}
Thus $s(X_i,X_j)$ is $0$ and
\begin{equation*}
b(X_i,X_j) = t(X_i,X_j) - h(X_i,X_j) + s(X_i,X_j) = 0.
\end{equation*}
\par
Summarising the cases:
\begin{equation}\label{eqn:bm_zeroexample}
b(X_i,X_j) = 
\begin{cases}
 1 & \text{if $i-j=-1$ or $i-j=-2$ } \\
-1 & \text{if $i-j=1$ or $i-j=2$ } \\
 0 & \text{otherwise.}
\end{cases}
\end{equation}
\par
Now we have calculated the based matrix for $\alpha_n$ we can determine
 whether it has any annihilating, core or complementary elements.
As $b(X_i,X_{i+1})$ is non-zero and $\n(X_i)=0$ for all $i$, $X_i$ cannot
 be an annihilating element or a core element. We now assume that there
 are an $i$ and a $j$ such that $X_i$ and $X_j$ form a complementary
 pair. Then
\begin{equation}\label{eqn:bm_zeroexample_complementary}
b(X_i,X) + b(X_j,X) = b(s,X) = 0  
\end{equation}
for all $X$. By substituting $X_{j-2}$ for $X$ and using
 \eqref{eqn:bm_zeroexample} we obtain
\begin{equation*}
b(X_i,X_{j-2}) = 1.
\end{equation*}
Similarly, by substituting $X_{j-1}$ for $X$ we get
\begin{equation*}
b(X_i,X_{j-1}) = 1.
\end{equation*}
By comparing these last two equations with
 \eqref{eqn:bm_zeroexample} we can conclude that $i-(j-2)$ and $i-(j-1)$
 are in the set $\lbrace -1, -2 \rbrace$. Thus
\begin{equation}\label{eqn:bm_zeroexample_ij1}
 i \equiv j+3 \text{(mod $n$)}.
\end{equation}
\par
Similarly, by substituting $X_{j+2}$ and $X_{j+1}$ for $X$ in
 \eqref{eqn:bm_zeroexample_complementary} and using
 \eqref{eqn:bm_zeroexample} we have
\begin{equation*}
b(X_i,X_{j+2}) = b(X_i,X_{j+1}) = -1.
\end{equation*}
Comparison with \eqref{eqn:bm_zeroexample} implies that
 $i-(j+2)$ and $i-(j+1)$ are in the set $\lbrace 1, 2 \rbrace$. In this
 case \begin{equation}\label{eqn:bm_zeroexample_ij2}
i \equiv j-3 \text{(mod $n$)}.
\end{equation}
\par
Combining \eqref{eqn:bm_zeroexample_ij1} and
 \eqref{eqn:bm_zeroexample_ij2} we have 
\begin{equation*}
j+3 \equiv j-3 \text{(mod $n$)}
\end{equation*}
which gives
\begin{equation*}
0 \equiv 6 \text{(mod $n$)}.
\end{equation*}
As we assumed that $n \geq 5$, this implies that if $X_i$ and $X_j$ are
 a complementary pair then $n$ is $6$.
\par
So if $n=5$ or $n \geq 7$, $M(\alpha_n)$ is primitive. In particular this
 means that for these cases $\rho(\alpha_n)$ is equal to $n$. This shows that
 the $\alpha_n$ are non-trivial and distinct.
\end{proof}
By the lemma we have an infinite number of examples of members of
$\baseSet{0}$. Thus we have the following theorem.
\begin{thm}
The set $\baseSet{0}$ is infinite in size.
\end{thm}
\section{Geometric interpretation of coverings of
 nanowords}\label{sec:geomcover}
In \cite{Turaev:2004} Turaev gave a geometric interpretation of the
covering operation on nanowords which we explain in more detail here.
\par
Given a nanoword we can construct the corresponding diagram of a curve on
its canonical surface (we descibed this towards the end of
Section~\ref{sec:virtual_strings}).
The curve can then be considered as a graph on the
surface. If the nanoword has rank $n$, the curve has $n$ crossings and
so the graph has $n$ vertices and $2n$ edges. We assign labels
$0,1,\dotsc,2n-1$ to the edges by assigning $0$ to the edge with the
base point and
then, starting from that edge, follow the curve assigning labels in
order to the edges as they are traversed.
Figure~\ref{fig:5xprecover} gives an example of such a labelled diagram
corresponding to the nanoword $\nanoword{ABCDBEDEAC}{baaaa}$. 
\begin{figure}[hbt]
\begin{center}
\includegraphics{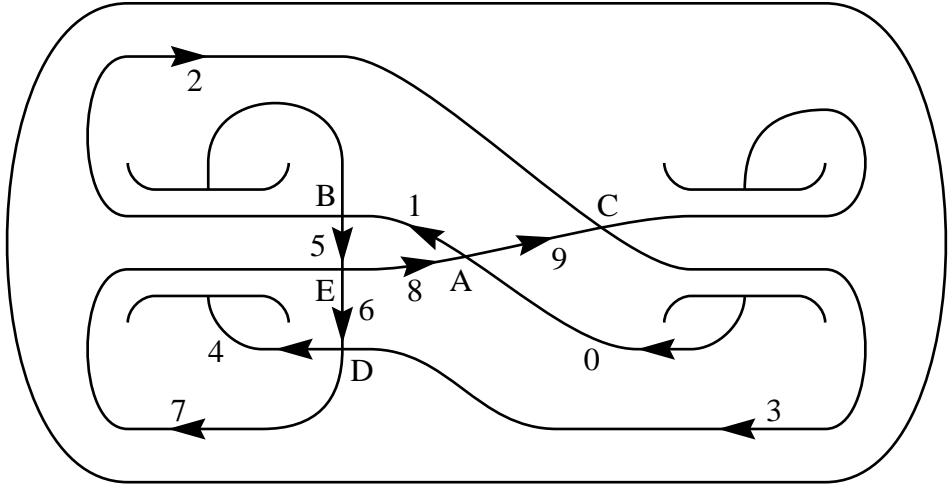}
\caption{A 5 crossing virtual string drawn on a closed genus 2 surface with edges labelled}
\label{fig:5xprecover}
\end{center}
\end{figure}
\par
By the construction of the canonical surface, cutting along the curve
splits the surface into one or more polygonal pieces. Each of the edges
of the polygons inherits a label and an orientation from the original
diagram. Obviously, we can reconstruct the canonical surface by glueing
edges with the same labels so that their orientations coincide.
Figure~\ref{fig:5xpoly} shows the result of this process on the diagram
in Figure~\ref{fig:5xprecover}.
\begin{figure}[hbt]
\begin{center}
\includegraphics{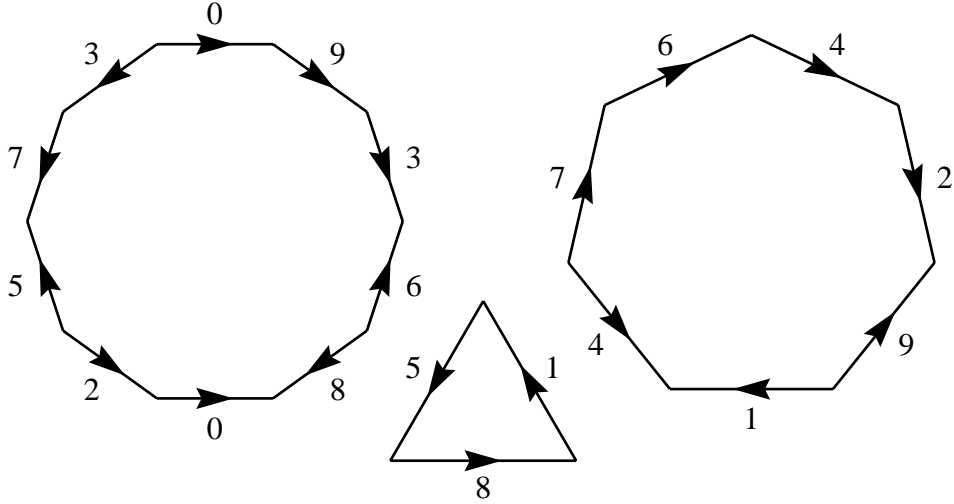}
\caption{Polygons with labelled oriented edges. Glueing edges with
 matching labels so that their orientations coincide gives the surface
 in Figure~\ref{fig:5xprecover}.}
\label{fig:5xpoly}
\end{center}
\end{figure}
\par
We now make an $m$-fold cover of the canonical surface (for $m$ an
integer greater than $1$).
We first make $m$ copies of the set of polygonal pieces and label the sets
$0,1,\dotsc,m-1$. We relabel the edges of each set as follows.
In the set of pieces labelled $i$, we relabel the edge
$j$ as $j_i$ if the orientation of the edge (relative to the polygon) is
anti-clockwise and relabel it $j_{(i-1)}$ ($i-1$ calculated modulo $n$) if
the orientation is clockwise. This labelling divides the edges into pairs. 
Figure~\ref{fig:5xpoly2} shows two copies of the polygons in
Figure~\ref{fig:5xpoly} labelled in the way described above.   
\begin{figure}[hbt]
\begin{center}
\includegraphics{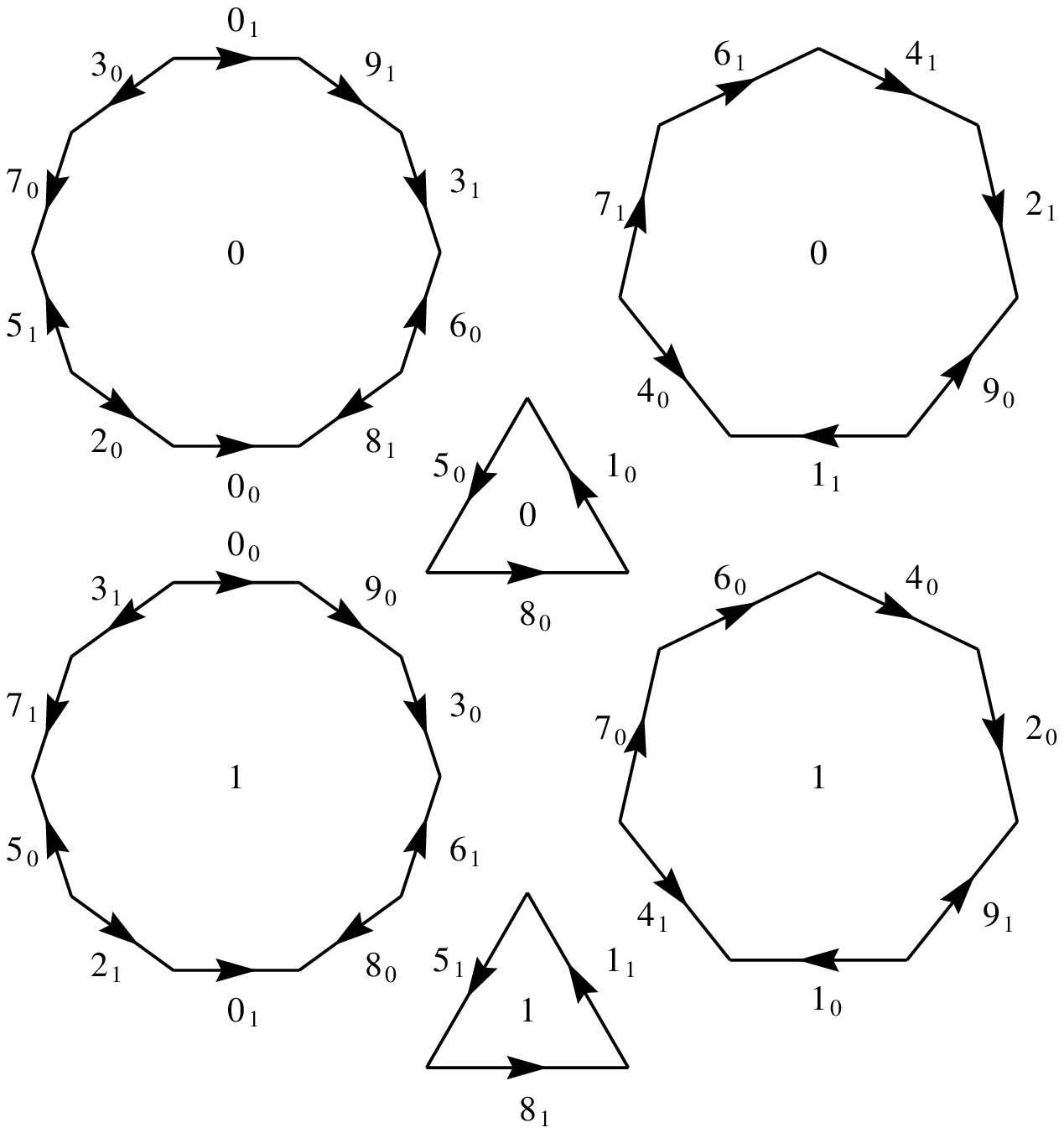}
\caption{Two copies of the polygons in Figure~\ref{fig:5xpoly} with
 edges relabelled}
\label{fig:5xpoly2}
\end{center}
\end{figure}
\par
By glueing the pairs of edges according to their
orientation we get a new closed surface. Figure~\ref{fig:5xcover} shows
the result of glueing the edges of Figure~\ref{fig:5xpoly2}.
\begin{figure}[hbt]
\begin{center}
\includegraphics{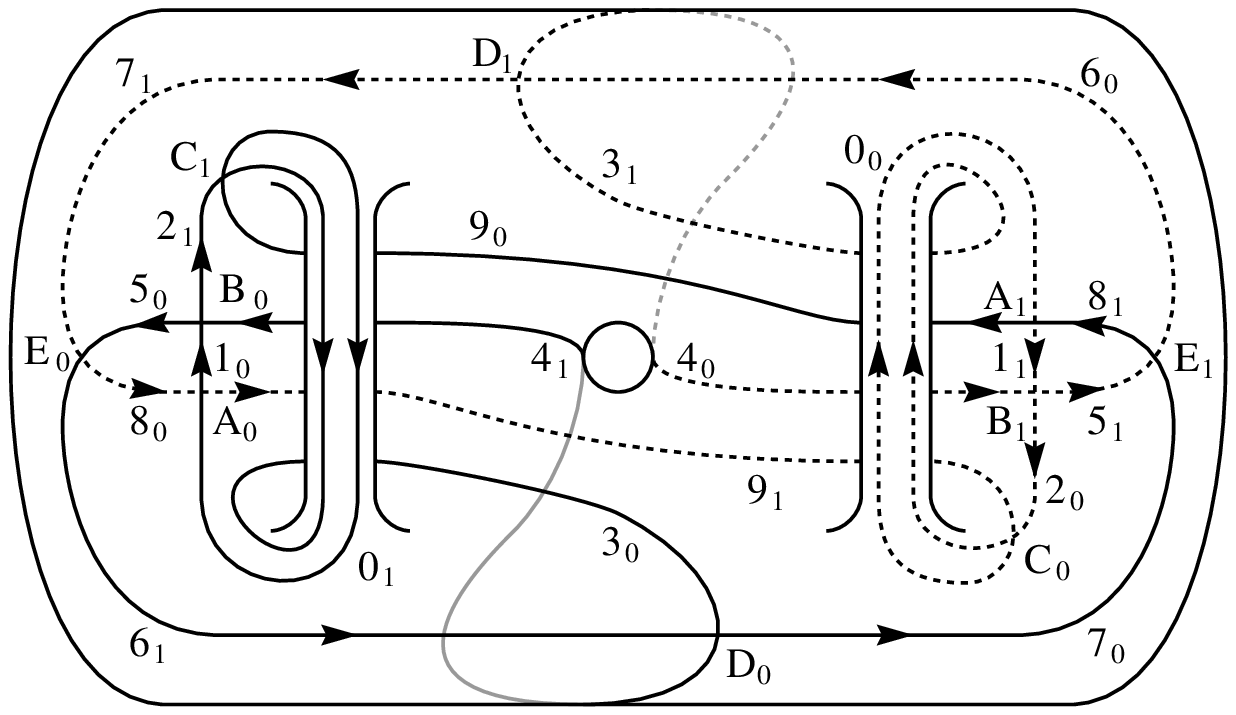}
\caption{The genus 3 surface constructed by glueing edges of the polygons in
 Figure~\ref{fig:5xpoly2}. Lines in light grey are on the underside
 of the surface. The glued edges form a 2 component virtual string. The
 components are distinguished by using dotted and solid lines.}
\label{fig:5xcover}
\end{center}
\end{figure}
\par
Note that for a given
edge $e$, the label of the polygon to the left of $e$, $l_e$ and the
label of the polygon to the right of $e$, $r_e$, satisfy the relation
\begin{equation}
r_e \equiv l_e + 1 \text{(mod $n$)}
\end{equation}
by construction.
\par
\begin{figure}[hbt]
\begin{center}
\includegraphics{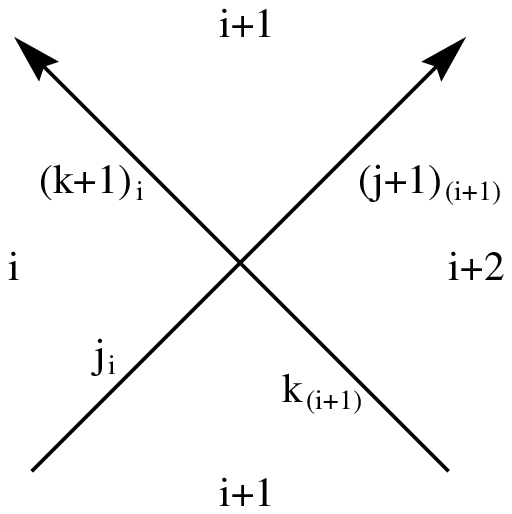}
\caption{Labels of polygons and edges around a vertex}
\label{fig:planelabel}
\end{center}
\end{figure}
The glued edges now form a graph on the constructed surface. 
In fact, each vertex of the graph is $4$-valent and the labels of the
polygons and edges meeting at a vertex have a fixed pattern as shown 
in Figure~\ref{fig:planelabel}.
Since there were $2n$ edges on the original surface, there are $2mn$
edges on the new surface. We now show that these edges form an
$m$-component virtual string.
\par
Consider the edge labelled $j_i$. It corresponds to the edge $j$ in the
original diagram. When we cross the vertex at the end of edge $j_i$, the
next edge corresponds to the edge $j+1$ (modulo $n$) in the original
diagram. If the edges at the vertex perpendicular to the edge $j_i$ are
oriented leftwards from the perspective of edge $j_i$ then the next edge
has subscript $i+1$ (modulo $m$). In other words, on the new surface the
edge $(j+1)_{i+1}$ follows the edge $j_i$. On the other hand, if the
edges at the vertex perpendicular to the edge $j_i$ are oriented
rightwards, then the edge $(j+1)_{i-1}$ follows the edge $j_i$. Both
these situations can be seen in Figure~\ref{fig:planelabel}.
\par
Consider the edge labelled $0_k$ for some $k$. This edge corresponds to
the edge $0$ in the original diagram. We now consider what happens to
this subscript as we traverse the curve in the original diagram. Each
time we pass through a crossing in the original diagram, this
corresponds to passing through a crossing in the new diagram. We have seen
that when we do this, passing through on one arc of the crossing
increases the subscript by one (modulo $n$) and passing through the
other decreases the subscript by one (modulo $n$).
When we get back to our starting point we have passed through each
crossing exactly twice. 
Therefore, when we get back to our
starting point, the net change on the edge label subscript is zero. This
means that in our new diagram a single component has exactly $2n$ edges
which are in one-to-one correspondance with the $2n$ edges in the
original 
diagram. Thus there are $m$ components in the covering and these
are indexed by the edge label subscript of edges corresponding to edge
$0$ in the original diagram. 
\par
The new diagram contains $mn$ crossings. By dropping the subscripts of
the edge labels around a vertex we
can see that these can be grouped in sets of $m$ crossings, each set
corresponding to a single crossing in the original diagram.
We can label the crossings in the new diagram by taking the label of the
corresponding crossing in the original diagram and adding a subscript
which is the label of the polygon to the left of both arcs in the
crossing. For example, this labelling scheme is used in
Figure~\ref{fig:5xcover}.
\par
The crossings in the new diagram can be divided
into two types. Those for which both arcs belong to the same component
and those for which the arcs belong to different components.
\par
Consider the single component containing the edge labelled $0_k$ for
some $k$ in the covering.
Then the edges of this component correspond to edges in the original
diagram and we can copy the edge label subscripts to the original
diagram under this correspondance.
For example, Figure~\ref{fig:5xpullback} shows the result of copying the
edge label subscripts of the single component containing $0_0$
 in Figure~\ref{fig:5xcover} back to the original diagram
 (Figure~\ref{fig:5xprecover}).
\begin{figure}[hbt]
\begin{center}
\includegraphics{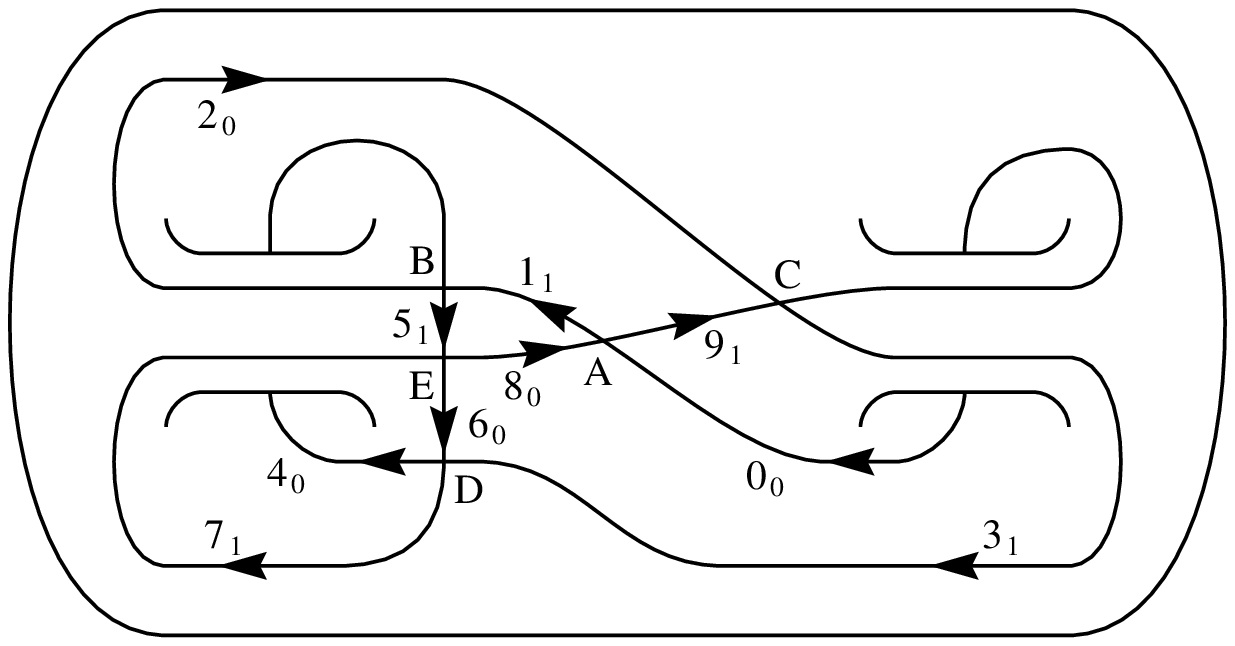}
\caption{A copy of Figure~\ref{fig:5xprecover} with edge label
 subscripts copied from the single component containing the edge $0_0$
 in Figure~\ref{fig:5xcover}.}
\label{fig:5xpullback}
\end{center}
\end{figure}
\par
By rotating a crossing we can orient it so that it looks like the
crossing on the left of Figure~\ref{fig:virtualcrossings}.
We define left and right edges of a crossing with respect to this
orientation.
\par
A necessary and sufficient condition for the component to
cross itself at a particular crossing in the new diagram is that the
edge label subscripts of the two right hand edges of the corresponding
crossing in the original diagram are the same (or, equivalently,
that the edge label subscripts of the two right left edges are the same).
This can be seen by considering Figure~\ref{fig:planelabel}.
If this condition is not met, the crossing arcs in the original diagram
correspond to arcs which go through different crossings in the new
diagram. These crossings are formed with arcs of a different
component or components.
\par
For example, in Figure~\ref{fig:5xpullback}, the crossings $B$, $C$ and
$D$ satisfy the condition that the subscript of both right hand edges are
the same.
In Figure~\ref{fig:5xcover} we can see that at each of the crossings
labelled $B_0$, $B_1$, $C_0$, $C_1$, $D_0$ and $D_1$, both arcs belong
to the same component.
On the other hand, the crossings $A$ and $E$ do not satisfy the
condition in Figure~\ref{fig:5xpullback}.
In Figure~\ref{fig:5xcover}, at each of the crossings $A_0$, $A_1$, $E_0$
and $E_1$ the two arcs belong to different components.
\par
For a given crossing $X$ in the original diagram, we define $\delta(X)$ to
be the subscript of the incoming right edge minus the subscript of the
outgoing right edge, modulo $m$. 
Then $X$ becomes a crossing where a component crosses itself in the new
diagram if and only if $\delta(X)$ is $0$.
If we remove a small neighbourhood of $X$ from the original diagram, the
curve is split into two sections: the left section of the curve
which starts at the left hand outgoing edge of $X$ and ends at the left
hand incoming edge of $X$; and the right section which starts at the
right hand outgoing edge of $X$ and ends at the right hand incoming edge
of $X$.
We write $r(X)$ for the number of arcs
crossing the right section from right to left minus the number of arcs
crossing the right section from left to right.
Then $\delta(X)$ is also equivalent, modulo $m$, to $r(X)$.
Now note that $\n(X)$ is equal to $r(X)$.
This implies that for $X$ to
be a self-crossing crossing in the new diagram, $\n(X)$ must be equal to
$0$ modulo $m$. Now this condition is the same as the condition used in
the construction of the \textcover{m} of a virtual string. If we just
consider a single component of our new diagram and remove all other
components we have a new virtual string and by the above discussion it
should be clear that this is a diagram of the \textcover{m} of our
original virtual string.
\par
By a similar argument it can be seen that the \textcover{0} of a
nanoword corresponds to the infinite cyclic cover of the associated
canonical surface. 
\par
Returning to the example in Figure~\ref{fig:5xcover}, we consider the
component drawn by the dotted line. Starting from the edge labelled $0_0$ we
follow the curve noting only the labels of crossings where the component
crosses itself. The result is the Gauss word $B_1C_0D_1B_1D_1C_0$. The
type of all three letters is $a$. This gives a nanoword which is
isomorphic to $\nanoword{BCDBDC}{aaa}$ which is the \textcover{2} of
$\nanoword{ABCDBEDEAC}{baaaa}$, the nanoword we started with.
\par
\begin{figure}[hbt]
\begin{center}
\includegraphics{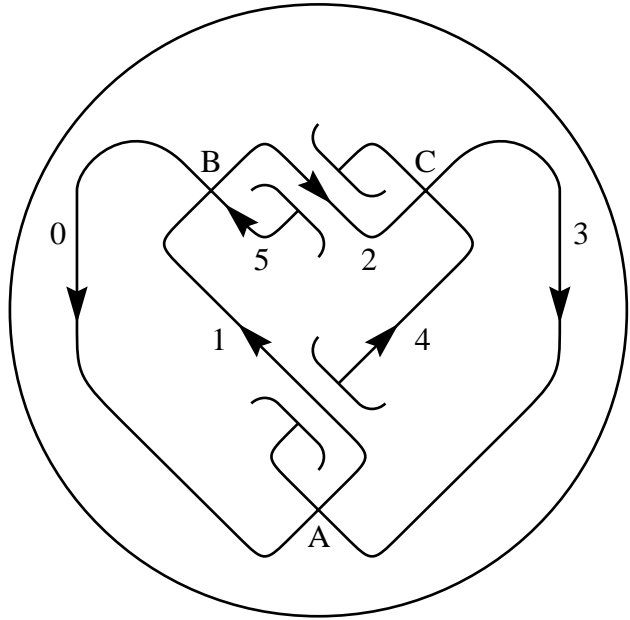}
\caption{A virtual string drawn on a closed genus 2 surface with edges labelled}
\label{fig:31precover}
\end{center}
\end{figure}
\begin{figure}[hbt]
\begin{center}
\includegraphics{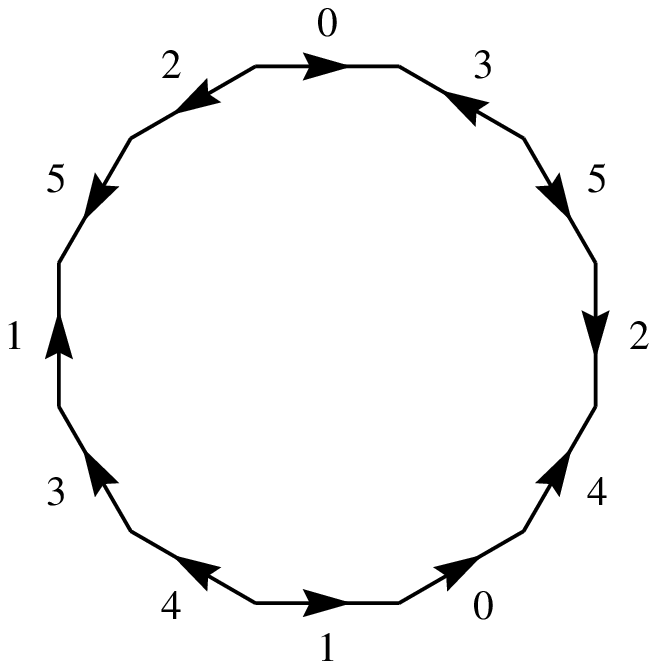}
\caption{A dodecagon with labelled oriented edges. Glueing edges with
 matching labels so that their orientations coincide gives the surface
 in figure \ref{fig:31precover}.}
\label{fig:31poly}
\end{center}
\end{figure}
\begin{figure}[hbt]
\begin{center}
\includegraphics{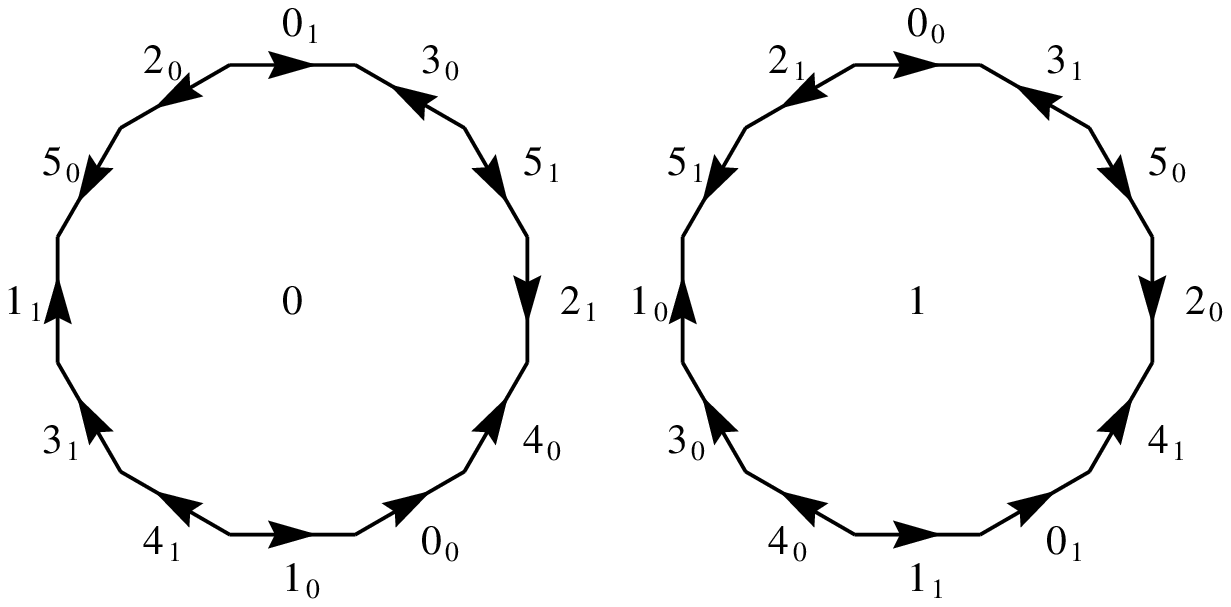}
\caption{Two copies of the dodecagon in Figure~\ref{fig:31poly} with
 edges relabelled}
\label{fig:2coverpoly}
\end{center}
\end{figure}
\begin{figure}[hbt]
\begin{center}
\includegraphics{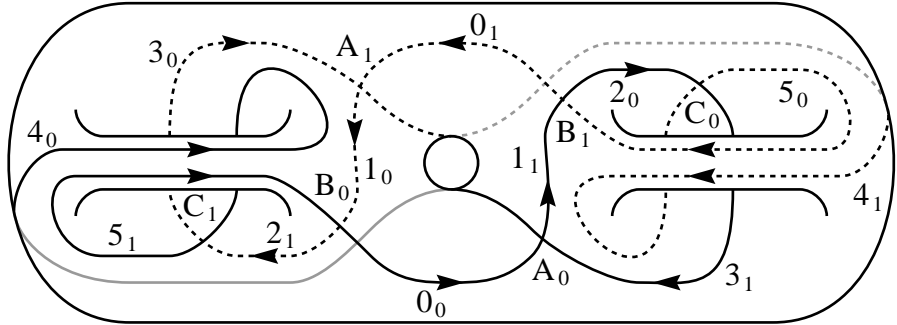}
\caption{The genus 3 surface constructed by glueing edges of the polygons in
 Figure~\ref{fig:2coverpoly}. Lines in light grey are on the underside
 of the surface. The glued edges form a 2 component virtual string. One
 component is drawn with a dashed line, the other with a solid line.}
\label{fig:2cover}
\end{center}
\end{figure}
\begin{figure}[hbt]
\begin{center}
\includegraphics{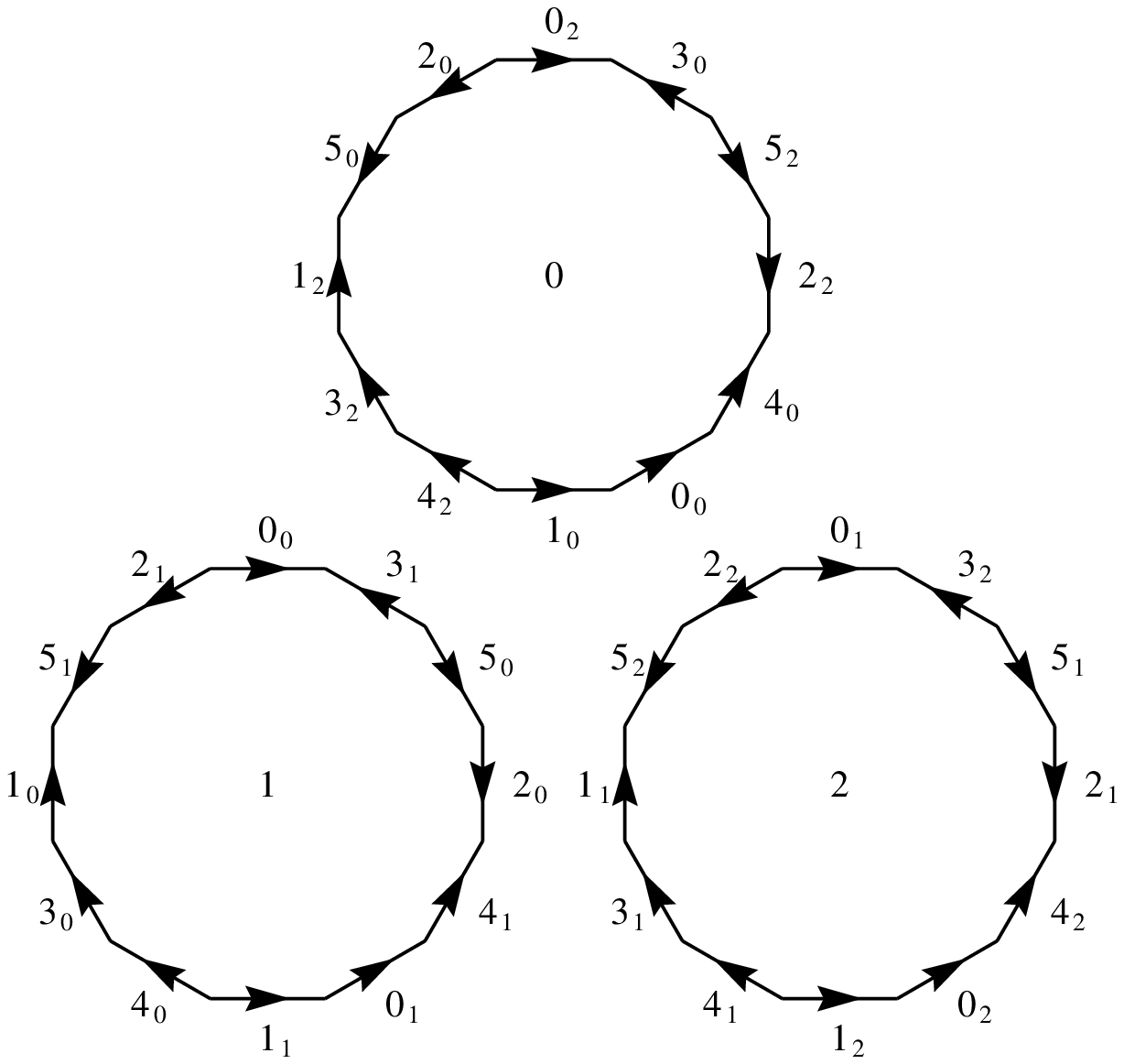}
\caption{Three copies of the dodecagon in Figure~\ref{fig:31poly} with
 edges relabelled}
\label{fig:3coverpoly}
\end{center}
\end{figure}
\begin{figure}[hbt]
\begin{center}
\includegraphics{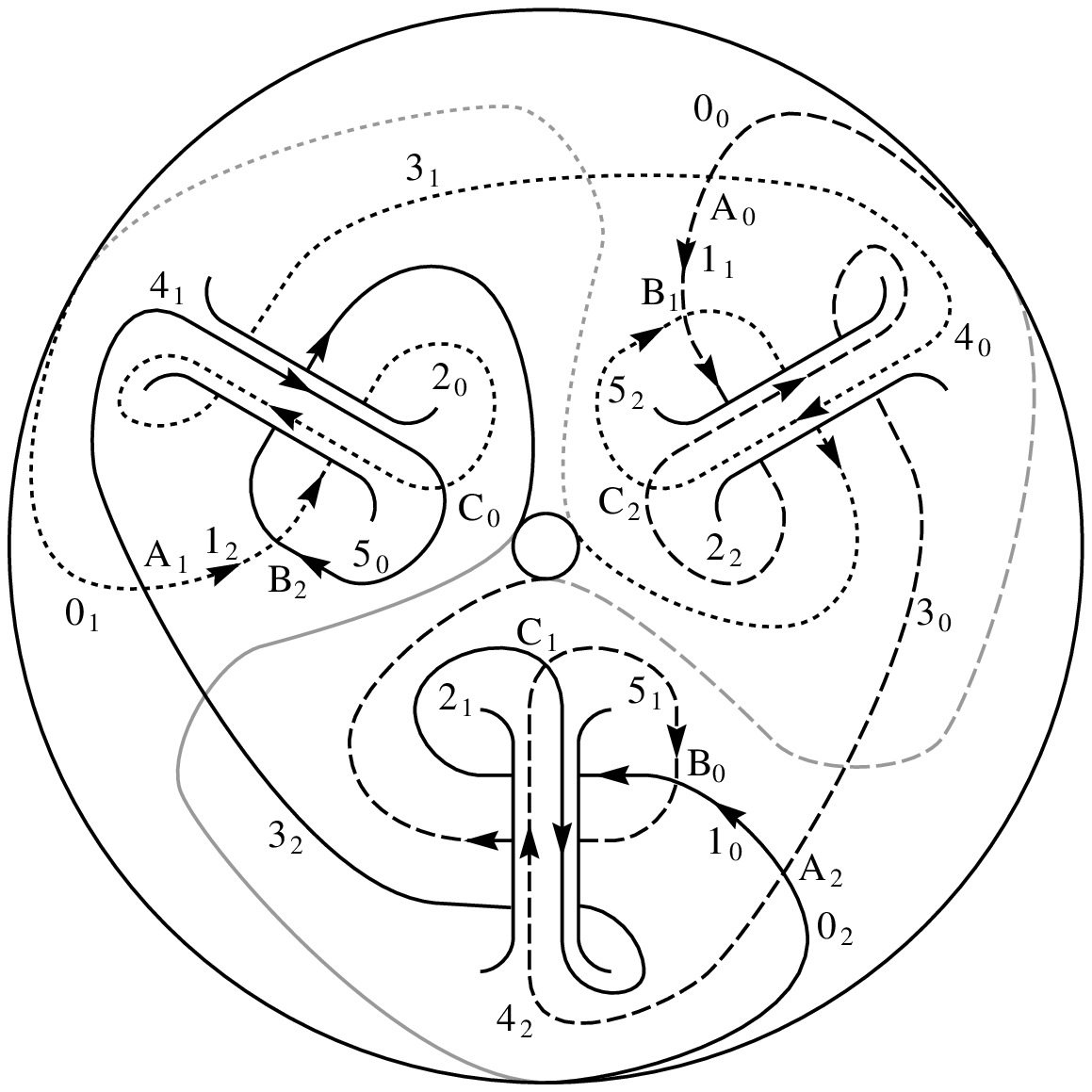}
\caption{The genus 4 surface constructed by glueing edges of the polygons in
 Figure~\ref{fig:3coverpoly}. Lines in light grey are on the underside
 of the surface. The glued edges form a 3 component virtual string. The
 components are distinguished by using dashed, dotted and solid lines.}
\label{fig:3cover}
\end{center}
\end{figure}
As further examples we construct the \textcover{2} and \textcover{3} of
the nanoword $\nanoword{ABCACB}{aaa}$ which is isomorphic to the
\textcover{2} in the previous example.
We have already calculated the coverings of this nanoword in Example
\ref{ex:simple_covering}. 
Figure~\ref{fig:31precover} gives a diagram of $\nanoword{ABCACB}{aaa}$
on a surface with the edges labelled. Figure~\ref{fig:31poly} shows the
result of cutting along the edges, in this case a single polygon. The
\textcover{2} is constructed from two copies of this polygon shown in
Figure~\ref{fig:2coverpoly}. Figure~\ref{fig:2cover} shows the result of
glueing. Similarly Figures~\ref{fig:3coverpoly} and \ref{fig:3cover}
show the \textcover{3} of $\nanoword{ABCACB}{aaa}$ before and after
glueing.
\section{Based matrices of composite nanowords}\label{sec:composite}
For two nanowords $\alpha$ and $\beta$,
it is possible to calculate the based matrix of the composite nanoword
$\alpha\beta$ from the based matrices of $\alpha$ and $\beta$ and the
types of the letters in $\alpha$ and $\beta$. In fact, in
\cite{Turaev:2004} Turaev gave the result of this calculation for graded
based matrices of open virtual strings (an open virtual string can be
defined as a virtual 
string with a fixed base point on the curve where no moves can involve
the base point). The statement of the result in Turaev's paper
contains an error, but once this is corrected, the result is the same
for the case of virtual strings. As Turaev did not give the calculation
in his paper, we give it here. 
 The main aim of this section is to prove
 Proposition~\ref{prop:compositeSameLetters} and
 Corollary~\ref{cor:primitive}.
\par
Using isomorphisms we can write $\alpha$ using letters
$W_1,W_2,\dotsc,W_m$ and $\beta$ using letters $X_1,X_2,\dotsc,X_n$. We
can write the based matrices of $\alpha$ and $\beta$ in the following form
\begin{equation*}
b_\alpha=
\begin{pmatrix}
0& -\transpose{\cvector{n_\alpha}} \\
\cvector{n_\alpha}& B_\alpha
\end{pmatrix},
b_\beta=
\begin{pmatrix}
0& -\transpose{\cvector{n_\beta}} \\
\cvector{n_\beta}& B_\beta
\end{pmatrix}.
\end{equation*}
\begin{prop}[Turaev]\label{prop:composite}
Using the notation above, the based matrix for the composite nanoword $\alpha\beta$
 has the form
\begin{equation*}
b_{\alpha\beta}=
\begin{pmatrix}
0& -\transpose{\cvector{n_\alpha}} & -\transpose{\cvector{n_\beta}}\\
\cvector{n_\alpha}& B_\alpha & D \\
\cvector{n_\beta}& -\transpose{D} & B_\beta
\end{pmatrix}
\end{equation*}
where the entries of $D$ are given by
\begin{equation*}
b_{\alpha\beta}(W_i,X_j) = 
\begin{cases}
0 & \text{if $\xType{W_i} = \xType{X_j} = a$ } \\
-n_\beta(X_j) & \text{if $\xType{W_i} = b, \xType{X_j} = a$ } \\
n_\alpha(W_i) & \text{if $\xType{W_i} = a, \xType{X_j} = b$ } \\
n_\alpha(W_i)-n_\beta(X_j) & \text{if $\xType{W_i} = \xType{X_j} = b$. }
\end{cases}
\end{equation*}
\end{prop}
\begin{proof}
As the letters $W_i$ and $X_j$ do not link in $\alpha\beta$ for all $i$ and all
 $j$, it is clear that $\n_{\alpha\beta}(W_i)$ equals $\n_\alpha(W_i)$ and
 $\n_{\alpha\beta}(X_j)$ equals $\n_\beta(X_j)$ for all $i$ and all $j$.
\par
We now consider $b_{\alpha\beta}(W_i,W_j)$. We use
 \eqref{eqn:matrix_elements} and the fact that for all $k$ and $l$,
 $h(W_k,X_l) = t(W_k,X_l)$. We have
\begin{align*}
b_{\alpha\beta}(W_i,W_j) =& t(W_i,W_j) - h(W_i,W_j) + 
\sum_{k \in G-\lbrace s\rbrace}t(W_i,k)h(W_j,k) - h(W_i,k)t(W_j,k) \\
= & t(W_i,W_j) - h(W_i,W_j) + \\
& \sum_{k}\left( t(W_i,W_k)h(W_j,W_k) - h(W_i,W_k)t(W_j,W_k) \right) + \\
& \sum_{k}\left( t(W_i,X_k)h(W_j,X_k) - h(W_i,X_k)t(W_j,X_k) \right) \\
= & b_\alpha(W_i,W_j) +
\sum_{k}\left( t(W_i,X_k)t(W_j,X_k) - t(W_i,X_k)t(W_j,X_k) \right) \\
= & b_\alpha(W_i,W_j).
\end{align*}
We can calculate $b_{\alpha\beta}(X_i,X_j)$ in the same way.
\par
We now calculate $b_{\alpha\beta}(W_i,X_j)$. We have
\begin{align*}
b_{\alpha\beta}(W_i,X_j) =& t(W_i,X_j) - h(W_i,X_j) + 
\sum_{k \in G-\lbrace s\rbrace}t(W_i,k)h(X_j,k) - h(W_i,k)t(X_j,k) \\
= &
\sum_{k}\left( t(W_i,W_k)h(X_j,W_k) - h(W_i,W_k)t(X_j,W_k) \right) + \\
& \sum_{k}\left( t(W_i,X_k)h(X_j,X_k) - h(W_i,X_k)t(X_j,X_k) \right).
\end{align*}
\par
Consider the first part.
If $\xType{X_j}$ is $a$ then $h(X_j,W_k) = t(X_j,W_k) = 0$ and 
\begin{equation*}
\sum_{k}\left( t(W_i,W_k)h(X_j,W_k) - h(W_i,W_k)t(X_j,W_k) \right) = 0.
\end{equation*}
If $\xType{X_j}$ is $b$ then $h(X_j,W_k) = t(X_j,W_k) = 1$ and 
\begin{align*}
\sum_{k}\left( t(W_i,W_k)h(X_j,W_k) - h(W_i,W_k)t(X_j,W_k) \right) & =
\sum_{k}\left( t(W_i,W_k) - h(W_i,W_k) \right) \\
& = \n(W_i).
\end{align*}
\par
Consider the second part.
If $\xType{W_i}$ is $a$ then $h(W_i,X_k) = t(W_i,X_k) = 0$ and 
\begin{equation*}
\sum_{k}\left( t(W_i,X_k)h(X_j,X_k) - h(W_i,X_k)t(X_j,X_k) \right) = 0.
\end{equation*}
If $\xType{W_i}$ is $b$ then $h(W_i,X_k) = t(W_i,X_k) = 1$ and 
\begin{align*}
\sum_{k}\left( t(W_i,X_k)h(X_j,X_k) - h(W_i,X_k)t(X_j,X_k) \right) & =
\sum_{k}\left( h(X_j,X_k) - t(X_j,X_k) \right) \\
& = -\n(X_j).
\end{align*}
\par
By combining the results of the calculations of the two parts the proof
 is complete.
\end{proof}
We now make some observations. If $Y$ is an annihilating element in
$M(\alpha\beta)$ then $Y$
must be an annihilating element in $M(\alpha)$ or in $M(\beta)$.
If $Y$ is a core element in $M(\alpha\beta)$ then $Y$
must be a core element in $M(\alpha)$ or in $M(\beta)$.
If $Y$ and $Z$ are complementary elements in $M(\alpha\beta)$ we have two
possibilities. The first case is that $Y$ and $Z$ both come from the same
component, say $\alpha$. In this case $Y$ and $Z$ are complementary elements
in $M(\alpha)$. The second case is that $Y$ and $Z$ come from different
components, say $Y$ comes from $\alpha$ and $Z$ comes from $\beta$. As $Y$ and
$Z$ are complementary we have $\n(Y)+\n(Z)=0$ and for all letters $K$ in
$\alpha\beta$
\begin{equation*}
b_{\alpha\beta}(Y,K) + b_{\alpha\beta}(Z,K) = b_{\alpha\beta}(s,K) = -\n(K).
\end{equation*}
Then $b_{\alpha\beta}(Y,Z)$ is $-\n(Z)$. However, by
Proposition~\ref{prop:composite} we have 
\begin{equation*}
b_{\alpha\beta}(Y,Z) =
\begin{cases}
0 &\text{if $\xType{Y} = \xType{Z} = a$} \\
-\n(Z) &\text{if $\xType{Y} = b, \xType{Z} = a$} \\
\n(Y) &\text{if $\xType{Y} = a, \xType{Z} = b$} \\
\n(Y)-\n(Z) &\text{if $\xType{Y} = \xType{Z} = b$.}
\end{cases}
\end{equation*}
\par
We now assume that $\xType{Z}=\xType{Y}$. Then from the above calculation
we have $\n(Y)=\n(Z)=0$. Furthermore, for $W$ in $\alpha$ we have
\begin{equation*}
b_{\alpha\beta}(Y,W) = -b_{\alpha\beta}(Z,W) - \n(W).
\end{equation*}
Using Proposition~\ref{prop:composite} to evaluate $b_{\alpha\beta}(Z,W)$ we get
\begin{equation*}
b_{\alpha\beta}(Y,W) =
\begin{cases}
-\n(W) &\text{if $\xType{Z} = \xType{W} = a$} \\
-\n(W) &\text{if $\xType{Z} = a, \xType{W} = b$} \\
0 &\text{if $\xType{Z} = b, \xType{W} = a$} \\
0 &\text{if $\xType{Z} = \xType{W} = b$.}
\end{cases}
\end{equation*}
Note that the value of $b_{\alpha\beta}(Y,W)$ does not depend on the type of
$W$. Similarly, for a letter $X$ in $x$, we can calculate
$b_{\alpha\beta}(Z,X)$. We have
\begin{equation*}
b_{\alpha\beta}(Z,X) = -b_{\alpha\beta}(Y,X) - \n(X)
\end{equation*}
and thus
\begin{equation*}
b_{\alpha\beta}(Z,X) = 
\begin{cases}
-\n(X) &\text{if $\xType{Y} = \xType{X} = a$} \\
-\n(X) &\text{if $\xType{Y} = a, \xType{W} = b$} \\
0 &\text{if $\xType{Y} = b, \xType{X} = a$} \\
0 &\text{if $\xType{Y} = \xType{X} = b$.}
\end{cases}
\end{equation*}
Again, the value of $b_{\alpha\beta}(Z,X)$ does not depend on the type of $X$.
\par
Thus if $Y$ in $\alpha$ and $Z$ in $\beta$ are complementary and both letters
are type $a$, then $Y$ is a core element in $M(\alpha)$ and $Z$ is a
core element in $M(\beta)$. If both letters are type $b$ then $Y$ is an
annihilating element in $M(\alpha)$ and $Z$ is an annihilating element
in $M(\beta)$.
\begin{prop}\label{prop:compositeSameLetters}
If the letters in $\alpha$ and $\beta$ all have the same type then
$\rho(\alpha\beta) \geq \rho(\alpha) + \rho(\beta)$.
\end{prop}
\begin{proof}
To calculate $\rho(\alpha\beta)$ we start with $M(\alpha\beta)$ and
 remove elements until we get to the primitive based matrix
 $P(\alpha\beta)$. From the above argument, any elements which can be
 removed from $M(\alpha\beta)$ correspond to elements that can be
 removed from $M(\alpha)$ or $M(\beta)$. Thus the result follows. 
\end{proof}
\begin{cor}\label{cor:primitive}
If the letters in $\alpha$ and $\beta$ all have the same type and
$M(\alpha)$ and $M(\beta)$ are primitive then
$\rho(\alpha\beta) = \rho(\alpha) + \rho(\beta)$.
\end{cor}
\begin{proof}
If $M(\alpha\beta)$ is not primitive, then by the above argument there exists at
 least one reducible element in $M(\alpha)$ or $M(\beta)$ which would contradict
 the assumption that $M(\alpha)$ and $M(\beta)$ are primitive.
\end{proof}
\section{Cablings of virtual strings}\label{sec:cable}
We can define a cabling of a virtual string in a similar way to cablings
of knots.
We can construct the \textcable{n} of a virtual string $\Gamma$ from a
diagram $D$ of $\Gamma$ embedded without virtual crossings in a surface
$S$. We pick an arbitrary point on $D$ which is not a double point
and cut $D$ at that point to get an arc in $S$.
We then replace the arc with $n$
 parallel copies of the arc. We label these arcs $0$ to $n-1$ from left to
 right as we travel along the curve according to its orientation. From
 the orientation of the curve, each arc has a start and an end. We join
 the end of arc $i$ to the beginning of arc $i+1$ for $i$ going from $0$
 to $n-2$. Finally we join the end of arc $n-1$ to the beginning of arc
 $0$  by crossing the other $n-1$ arcs. We call the resulting
 diagram $\cable{D}{n}$ and describe it as the \textcable{n} of $D$.
\par
\begin{figure}[hbt]
\begin{center}
\includegraphics{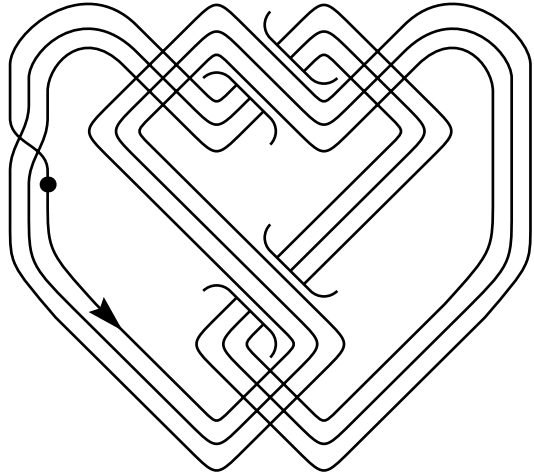}
\caption{The \textcable{3} of a virtual string}
\label{fig:cabling31}
\end{center}
\end{figure}
To illustrate the result of this process, Figure~\ref{fig:cabling31} 
 shows the \textcable{3} of the diagram in Figure~\ref{fig:diagram31}
 after the virtual crossings have been replaced by 
 handles. The point on the curve marked with a blob is where the $3$
 arcs have been joined to each other according to the description given
 above.
\par
Note that in this construction, each crossing in the original virtual
string diagram becomes $n^2$ crossings in the \textcable{n}. We also get
an additional $n-1$ crossings when we join up the arcs. Thus the number
of crossings in the \textcable{n} of a $k$ crossing virtual string
diagram is $kn^2 + n - 1$.
\par
We picked an arbitrary point on $D$ to join the arcs.
It is easy to check, using the flattened Reidemeister moves,
that we get a homotopic virtual string diagram, no matter which
 non-double point on $D$ we pick. It is also easy to check that for
 homotopically equivalent virtual string diagrams $D$ and $D^\prime$ the
 \textcable{n}s $\cable{D}{n}$ and $\cable{D^\prime}{n}$ are also
 homotopic. Thus, if we define the \textcable{n} of $\Gamma$ as the
 virtual string represented by $\cable{D}{n}$ for some diagram $D$
 representing $\Gamma$, the \textcable{n} of $\Gamma$ is well-defined
 and we write it $\cable{\Gamma}{n}$. In particular we note that
 invariants of cables of virtual strings can be used to distinguish
 virtual strings themselves.
\par
We remark that Turaev defined cables of virtual strings in
 \cite{Turaev:2004}. He defines cables as the result of ``Adams
 operations''. He noted that this is a well-defined operation on virtual
 strings. He also gave a relation between the $u$-polynomial of a
 virtual string and the $u$-polynomial of its \textcable{n}. We discuss
 this further below. 
\par
After we had written this section we discovered that Kadokami had
 defined cables of virtual strings in the same way as we have in
 \cite{Kadokami:Non-triviality}. He calls 
the result of the construction the $n$-parallelized projected virtual
knot diagram of $D$. Although Kadokami made this definition, he
did not examine any of the properties that we give here.
\par
We can define the \textcable{n} of a nanoword $\alpha$ in the following
 way.
We first construct a diagram $D$ corresponding to $\alpha$ on a surface
 $S$ such that $D$ has no virtual crossings.
Each crossing of $D$ is labelled with a letter from $\alpha$.
The nanoword $\alpha$ gives us a base
 point on $D$ which we use to construct $\cable{D}{n}$ as explained above. Using
 the base point we can then write a nanoword which represents
 $\cable{D}{n}$. We define this nanoword to be the \textcable{n} of
 $\alpha$.
\par
We examine this construction more closely.
We start by labelling the crossings of the \textcable{n} of $D$.
We have already labelled each arc $0$ to $n-1$ going from left to
right. The crossing points which are created at the points where we join
the $i$th arc to the $(i+1)$th arc are labelled $C_i$ (for $i$ running
from $0$ to $n-2$). For the $n\times n$ crossings derived
from a single crossing $A$ in $\alpha$ we adopt the following naming
scheme. 
We note that each crossing in $D$ has a type ($a$ or $b$) which is
 specified in $\alpha$.
If $A$ is of type $a$, we
simply label the crossings $A_{i,j}$ where the crossing is the
intersection of the $i$th arc coming from the left and the $j$th arc
coming from the right. 
If $A$ is of type $b$, we again label the crossings $A_{i,j}$. However
this time the crossing is the intersection of the $(i-1)$th arc
(calculating in $\cyclic{n}$) coming from the left and the $j$th arc coming
from the right. Figure~\ref{fig:cablelabel} shows how the arcs and
crossings are labelled for a \textcable{3}.
\begin{figure}[hbt]
\begin{center}
\includegraphics{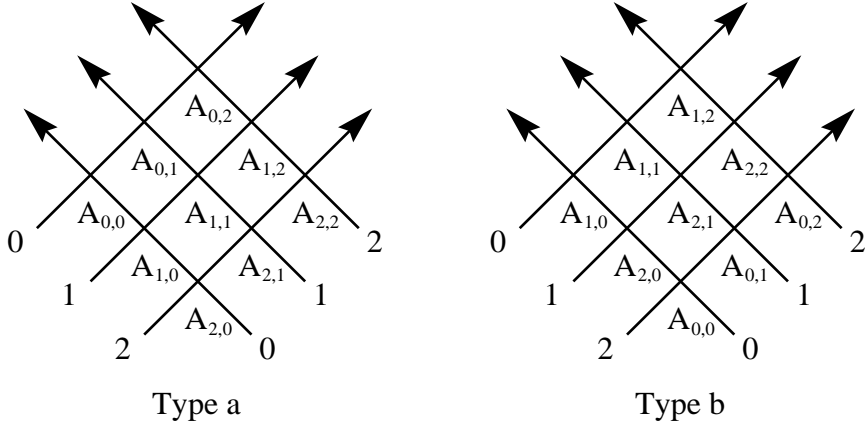}
\caption{Labelling crossings in a \textcable{3}}
\label{fig:cablelabel}
\end{center}
\end{figure}
\par
The types of the crossings $A_{i,j}$ in the cable can be determined
from $i$, $j$ and the type of $A$. If $A$ is of type $a$, $A_{i,j}$ is
of type $a$ if $i$ is less than or equal to $j$ and of type $b$ if $i$
is greater than $j$. If $A$ is of type $b$, writing $d$ for $i-1$
calculated in $\cyclic{n}$, $A_{i,j}$ is of type $a$ if $j$ is greater
than $d$ and of type $b$ if $j$ is less than or equal to $d$.
The crossings $C_i$ are all of type $a$. 
\par
Using this labelling scheme we can calculate the nanoword that represents
the cable of a virtual string directly from a nanoword $\alpha$ representing
the virtual string in a mechanical way. For an integer $i$ between $0$
and $n-1$ inclusive we define $w_i$ to be a copy of $\alpha$ with every
letter $A$ replaced by $n$ letters as follows. If $A$ has type $a$,
replace the first occurence of $A$ with
\begin{equation*}
A_{i,0}A_{i,1}\dotso A_{i,n-1}
\end{equation*}
and the second occurence of $A$ with
\begin{equation*}
A_{n-1,i}A_{n-2,i}\dotso A_{0,i}.
\end{equation*}
If $A$ has type $b$, replace the first occurence of $A$ with
\begin{equation*}
A_{0,i}A_{n-1,i}A_{n-2,i}\dotso A_{1,i}
\end{equation*}
and the second occurence of $A$ with
\begin{equation*}
A_{i+1,0}A_{i+1,1}\dotso A_{i+1,n-1}
\end{equation*}
where $i+1$ is calculated in $\cyclic{n}$.
Then the \textcable{n} of $\alpha$ is given by
\begin{equation*}
w_0C_0w_1C_1\dotso C_{n-3}w_{n-2}C_{n-2}w_{n-1}C_{n-2}C_{n-3}\dotso C_1C_0
\end{equation*}
where the types of the letters $A_{i,j}$ and letters $C_i$ are defined as above.
\par
In this way the cabling operation can be defined for nanowords
representing virtual strings without reference to diagrams. We denote
the \textcable{n} of a nanoword $\alpha$ by $\cable{\alpha}{n}$. We give an
example of calculating a \textcable{2} of a nanoword.
\par
\begin{ex}\label{ex:2cable}
Consider $\alpha$, the nanoword $\nanoword{XYXZYZ}{abb}$. Then $w_0$ is given by
\begin{equation*}
X_{0,0}X_{0,1}Y_{0,0}Y_{1,0}X_{1,0}X_{0,0}Z_{0,0}Z_{1,0}Y_{1,0}Y_{1,1}Z_{1,0}Z_{1,1},
\end{equation*}
$w_1$ is given by
\begin{equation*}
X_{1,0}X_{1,1}Y_{0,1}Y_{1,1}X_{1,1}X_{0,1}Z_{0,1}Z_{1,1}Y_{0,0}Y_{0,1}Z_{0,0}Z_{0,1}
\end{equation*}
and the \textcable{2} of $\alpha$, $\cable{\alpha}{2}$, is
 $w_0C_0w_1C_0$. The letters $X_{0,0}$, $X_{0,1}$, $X_{1,1}$, $Y_{1,1}$,
 $Z_{1,1}$ and $C_0$ are all of type $a$ and the remaining letters are
 of type $b$.
\end{ex}
\par
We can calculate the $u$-polynomial of a cable directly from the
$u$-polynomial of the original virtual string.
We note that in Section~5.4 of Turaev's paper \cite{Turaev:2004}, there
is a statement of the relationship between the $u$-polynomials of a
cable and the original virtual string. However there is an error in the
statement and no proof is given. We give the correct statement and
provide a proof here.
\begin{thm}\label{thm:cableUPoly}
For a virtual string $\Gamma$
the $u$-polynomial of the \textcable{n} of $\Gamma$ is given by
\begin{equation}\label{eqn:cableUPoly}
u_{\cable{\Gamma}{n}}(t) = n^2u_{\Gamma}(t^n).
\end{equation}
\end{thm}
\begin{proof}
Recall that the $u$-polynomial of a virtual string $\Gamma$ is given by
\begin{equation*}
u_{\Gamma}(t) = \sum_{k \geq 1}u_k(\Gamma)t^k.
\end{equation*}
Substituting this into \eqref{eqn:cableUPoly} and moving $n^2$
 inside the sum, we get
\begin{equation*}
u_{\cable{\Gamma}{n}}(t) = \sum_{k \geq 1}n^2u_k(\Gamma)t^{nk}.
\end{equation*}
We now prove this equation.
\par
We construct a diagram $D$ of $\Gamma$ in some surface $S$ so that $D$
 has no virtual crossings. We then construct $\cable{D}{n}$ the
 \textcable{n} of $D$.
\par
It is sufficient to prove the following two facts.
Firstly, for all $n^2$ crossings $A_{i,j}$ in
 $\cable{D}{n}$ coming from a crossing $A$ in $D$, $\n(A_{i,j})$ is
 equal to $n \n(A)$.
Secondly, for all $n-1$ crossings $C_i$ added by joining the arcs of the
 cable, we have $\n(C_i)=0$.
\par
To show both facts we recall that $\n(X)$ is the homological intersection
 number of the loop starting at $X$ and the curve $\cable{D}{n}$.
 The whole cable $\cable{D}{n}$, for the purposes of computing
 the homological intersection number, can be considered equivalent to
 $n$ parallel copies of the original curve $D$. Similarly,
 any loop in the cable can be considered to be equivalent to zero or
 more parallel copies of $D$ and possibly a loop in $\alpha$ starting at a
 crossing in $\alpha$.
\par
We first consider a crossing $A_{i,j}$ derived from a crossing $A$ in
 $D$. The loop at $A_{i,j}$ is equivalent to the loop at $A$ in $D$ and 
 $p$ parallel copies of $D$. Here $p$ is equal to $j-i$, calculating in
 $\cyclic{n}$. Then
\begin{equation*}
\n(A_{i,j}) = n \left(p\;b_D(s,s) + b_D(A,s) \right).
\end{equation*}
As $b_D(s,s) = 0$ and $b_D(A,s)=\n(A)$, we have 
\begin{equation*}
\n(A_{i,j}) = n \n(A).
\end{equation*}
We now consider a crossing $C_i$. In this case the loop at $C_i$ is
 equivalent to $(n-1)-i$ parallel copies of $D$. In this case we have
\begin{align*}
\n(C_i) = & n \left(\left( (n-1) - i \right) b_D(s,s) \right) \\
 = & 0 
\end{align*}
and the proof is complete.
\end{proof}
\begin{cor}
If $\Gamma$ is a virtual string with non-trivial $u$-polynomial, the
 family of virtual strings $\{\Gamma,\cable{\Gamma}{n+1}|n \in \Zpos\}$ are
 all mutually homotopically distinct.
\end{cor}
\begin{proof}
As $\Gamma$ has a non-trivial $u$-polynomial, the degree of the
 $u$-polynomial is $k$ for some non-zero positive integer $k$. By the
 theorem, the degree of the $u$-polynomial of $\cable{\Gamma}{n}$ is
 $nk$. Thus the $u$-polynomials of the family are all different and so
 the virtual strings themselves must all be homotopically distinct. 
\end{proof}
As we can consider both cablings and coverings as maps from the set of
virtual strings into itself, it makes sense to consider their composition.
Using our notation we can write the
\textcover{r} of the \textcable{n} of a virtual string $\Gamma$ as
 $\cover{(\cable{\Gamma}{n})}{r}$ and the \textcable{n} of the \textcover{r}
 of $\Gamma$ as $\cable{(\cover{\Gamma}{r})}{n}$. In general these are not
 necessarily the same. An example is sufficient to show this.
\par
\begin{ex}
We consider the virtual string $\Gamma$ represented by the nanoword
 $\nanoword{XYXZYZ}{abb}$. We consider the case where $n$ and $r$ are both
 $2$. 
Then $\cover{\Gamma}{2}$ is represented by $\nanoword{XX}{a}$ which is homotopically
 trivial. Thus $\cable{(\cover{\Gamma}{2})}{2}$ is also trivial. On the other
 hand, we calculated $\cable{\Gamma}{2}$ in Example \ref{ex:2cable}.
 It is easy to check (or see the discussion below) that this nanoword is
 fixed under the \textcover{2} map. That is,
 $\cover{(\cable{\Gamma}{2})}{2}$ is equal to $\cable{\Gamma}{2}$. As the
 $u$-polynomial for $\Gamma$ is $t^2-2t$, by
 Theorem~\ref{thm:cableUPoly}, the $u$-polynomial for
 $\cable{\Gamma}{2}$ is 
 $4t^4-8t^2$. Thus $\cable{\Gamma}{2}$ is non-trivial and
 $\cover{(\cable{\Gamma}{2})}{2}$ is not homotopic to
 $\cable{(\cover{\Gamma}{2})}{2}$.
\end{ex}
However, we do have a relationship between cablings of coverings and
coverings of cables.
\begin{thm}
For any virtual string $\Gamma$, any positive integer $n$ and any
 non-negative integer $r$, 
$\cable{(\cover{\Gamma}{k})}{n}$ and $\cover{(\cable{\Gamma}{n})}{r}$
 are homotopic. Here $k$ is zero if $r$ is zero and $k$ is
 $r/d$ otherwise, where $d$ is the greatest common divisor of $n$ and
 $r$. 
\end{thm}
\begin{proof}
It is enough to show that for some nanoword $\alpha$ representing $\Gamma$, 
$\cable{(\cover{\alpha}{k})}{n}$ and $\cover{(\cable{\alpha}{n})}{r}$ are
 isomorphic.
\par
First we consider $\cable{(\cover{\alpha}{k})}{n}$. 
We note that because of the mechanical way in which we can calculate the
 cable of a nanoword, we can calculate this nanoword from $\cable{\alpha}{n}$.
We can do this by deleting all letters
 $X_{i,j}$ in the nanoword $\cable{\alpha}{n}$ which came from a letter $X$ in $\alpha$
 such that $k$ does not divide $\n(X)$. Thus the $X_{i,j}$ that are
 retained from $\cable{\alpha}{n}$ are those such that $\n(X)$ is in $k\Z$.
\par
We now consider $\cover{(\cable{\alpha}{n})}{r}$. We first note that the
 $n-1$ crossings $C_i$ added to join up the $n$ arcs in the cabling all
 have $\n(C_i)=0$. Thus they are not deleted when we take the
 covering. We then note that any crossing $X$ in $\alpha$ gets transformed
 into $n^2$ crossings $X_{i,j}$ by the cabling 
 operation. As we noted in the proof of Theorem~\ref{thm:cableUPoly},
 $\n(X_{i,j})$ is equal to $n \n(X)$.
 If we can show that the letters $X_{i,j}$ that are retained after taking
 the \textcover{r} have the property that $\n(X)$ is in $k\Z$ then we have
 shown that $\cable{(\cover{\alpha}{k})}{n}$ and $\cover{(\cable{\alpha}{n})}{r}$ are
 isomorphic.
\par
If $r$ is zero then the letters $X_{i,j}$ appear in
 $\cover{(\cable{\alpha}{n})}{r}$ if and only if $n \n(X)$ is
 zero. As $n$ is non-zero, this can only happen if $\n(X)$ is
 zero. Thus $\n(X)$ is in $k\Z$.
\par
If $r$ is non-zero then the letters $X_{i,j}$ appear in
 $\cover{(\cable{\alpha}{n})}{r}$ if and only if $n \n(X)$ is
 divisible by $r$ or $\n(X)$ is zero. If $\n(X)$ is zero then $\n(X)$ is
 in $k\Z$. If $r$ divides $n \n(X)$ then, as $k$ divides $r$ and
 $k$ is coprime to $n$, $k$ must divide $\n(X)$. In other words $\n(X)$
 is in $k\Z$ and the proof is complete.
\end{proof}
We note that all cablings of virtual strings are
fixed points under some non-trivial covering.
We have seen that for a crossing $X_{i,j}$ in $\cable{\alpha}{n}$ derived from a
 crossing $X$ in $\alpha$,  $\n(X_{i,j})$ is equal to $n\n(X)$. 
Thus $\alpha$ is
 fixed under the \textcover{r} map if and only if $\cable{\alpha}{n}$ is
 fixed under the \textcover{rn} map. This means that, $\Gamma$ is in
 $\baseSet{r}$ if and only if $\cable{\Gamma}{n}$ is in $\baseSet{rn}$. In
 particular $\cable{\Gamma}{n}$ is in $\baseSet{n}$ as $\baseSet{1}$ contains
 all virtual strings.
\par
We now calculate the based matrix of an \textcable{n} of a virtual string
$\Gamma$. As before we take a labelled diagram $D$ of $\Gamma$ and
construct a labelled diagram of the cable $\cable{D}{n}$. We use the
notation for crossing labels that we used above.
\par
Recall that for the purposes of calculating the 
homological intersection number the loop at $A_{i,j}$ is equivalent to
the loop at $A$ in $D$ and $j-i$ parallel copies of $D$
(calculating in $\cyclic{n}$). The loop at $C_i$ is equivalent to
$(n-1)-i$ parallel copies of $D$. 
\par
We can calculate the based matrix of the \textcable{n} as follows:
\begin{equation*}
b(C_i,C_j) = (n-1-i)\left( (n-1-j) b_D(s,s) \right) = 0,
\end{equation*}
\begin{equation*}
b(A_{i,j},C_k) = (n-1-k)\left(b_D(A,s) + (j-i) b_D(s,s) \right)
= (n-1-k)\n(A)
\end{equation*}
and
\begin{align*}
b(A_{i,j},B_{k,l}) = & (l-k) \left( b_D(A,s) + (j-i) b_D(s,s) \right)
+ (j-i)b_D(s,B) + b_D(A,B) \\
=& b_D(A,B) + (l-k)\n(A) - (j-i)\n(B). 
\end{align*}
\begin{thm}
If $\Gamma$ and $\Gamma^\prime$ are virtual strings such that $P(\Gamma)$ and
 $P(\Gamma^\prime)$ are isomorphic, then $P(\cable{\Gamma}{n})$ is isomorphic to
 $P(\cable{\Gamma^\prime}{n})$ for all $n$.
\end{thm}
\begin{proof}
Let $\alpha$ be a nanoword representing $\Gamma$.
When we calculate $P(\alpha)$ from $M(\alpha)$ we make a series of reductions
 removing one or two elements at each step. For each such step we will show
 that there is an equivalent set of reduction moves on $M(\cable{\alpha}{n})$
 which allows us to remove $n^2$ or $2n^2$ elements. We call the matrix
 derived from $M(\cable{\alpha}{n})$ in this way
 $Q(\cable{\alpha}{n})$. Note that $Q(\cable{\alpha}{n})$ is not
 necessarily $P(\cable{\alpha}{n})$. 
\par
Assume that $A$ is an annihilating element in $M(\alpha)$. Then $\n(A)$ is $0$
 and $b(A,B)=0$ for all $B$ in $\alpha$. For each crossing $A_{i,j}$ in
 $\cable{\alpha}{n}$ which is derived from $A$ we have
\begin{equation*}
b(A_{i,j},C_k) = 0
\end{equation*}
and
\begin{equation*}
b(A_{i,j},B_{k,l}) =  - (j-i)\n(B) = -(j-i)\n(B_{k,l}). 
\end{equation*}
We note that values of $j-i$ run from $0$ to $n-1$ and there are
 $n$ pairs $(i,j)$ for which $j-i$ is the same.
We thus have $n$ elements $X_i$ derived from $A$ such that
\begin{equation*}
b(X_i,B_{k,l}) =  -i\n(B_{k,l}) 
\end{equation*}
for each $i$ running from $0$ to $n-1$.
Then the $n$ $X_0$ elements are all annihilating elements in
 $M(\cable{\alpha}{n})$.
We can pair all the remaining $X_i$ elements with $X_{n-i}$
 elements. Note that when $n$ is even, we pair the $X_{n/2}$ elements
 with themselves, but as there are an even number of such elements this
 is always possible. These pairs are complementary elements in
 $M(\cable{\alpha}{n})$ because
\begin{equation*}
b(X_i,Y) + b(X_{n-i},Y) =  -(i + n - i)\n(Y) = -n\n(Y) = b(s,Y)
\end{equation*}
for all crossings $Y$ corresponding to crossings in $\alpha$ and
\begin{equation*}
b(X_i,C_j) + b(X_{n-i},C_j) = 0 = b(s,C_j)
\end{equation*}
for all $j$.
\par
If $A$ is a core element of $M(\alpha)$ we can make a similar calculation. We
 find that we have $n$ core elements in $M(\cable{\alpha}{n})$ which are
 derived from $A$. We also find that the rest of the crossings derived
 from $A$ can be paired to form complementary pairs in
 $M(\cable{\alpha}{n})$.
\par
If $A$ and $B$ are complementary elements in $M(\alpha)$ then, we can pair
 crossings derived from $A$ and $B$ in a particular way to make
 complementary pairs in $M(\cable{\alpha}{n})$. Again the calculation is
 similar.
\par
We now note that $Q(\cable{\alpha}{n})$ can be defined algebraically in
 terms of $P(\alpha)$. That is, we can define $Q(\cable{\alpha}{n})$ in terms of
 $P(\alpha)$ without reference to $\alpha$ itself. We do this as follows.
\par
We define a based matrix $Q(P(\alpha),n)$ which is a triple $(G_Q,s,b_Q)$
 from the based matrix $P(\alpha)$ with triple $(G,s,b)$.
$G_Q$ consists of $s$, 
$n^2$ elements $X_{i,j}$ for each element $X$ of
$G-\lbrace s \rbrace$ in $P(\alpha)$ and $n-1$ elements $C_i$ ($i$ running
 from $0$ to $n-1$).
We define $\n(X_{i,j}) = n\n(X)$ and $\n(C_i)=0$. 
We then define
\begin{equation*}
b_Q(X_{i,j},Y_{k,l}) = b(X,Y) + k\n(X) - i\n(Y),
\end{equation*}
\begin{equation*}
b_Q(X_{i,j},C_k) = (k+1)\n(Y)
\end{equation*}
and
\begin{equation*}
b_Q(C_i,C_j) = 0.
\end{equation*}
It is easy to check that the based matrix $Q(P(\alpha),n)$ is isomorphic to
 $Q(\cable{\alpha}{n})$.  
\par
Now if $P(\Gamma^\prime)$ is
 isomorphic to $P(\Gamma)$ then $Q(\cable{\Gamma^\prime}{n})$ is isomorphic to
 $Q(\cable{\Gamma}{n})$ and thus the result follows.
\end{proof}
\par
The implication of this theorem is that there is no benefit in
calculating based matrix derived invariants for cables of virtual
strings in order to distinguish the virtual strings themselves.
\par
We note that in $Q(\cable{\alpha}{n})$, for any $Y$ not equal to $C_j$ for some $j$,
\begin{equation*}
b(C_i,Y) + b(C_{n-i-1},Y) = -n \n(Y) = b(s,Y)
\end{equation*}
and
\begin{equation*}
b(C_i,C_j) + b(C_{n-i-1},C_j) = 0 = b(s,C_j).
\end{equation*}
Thus $C_i$ and $C_{n-i-1}$ form a complementary pair in
$Q(\cable{\alpha}{n})$ if $i$ does not equal $n-i-1$. If $n$ is odd then we
can pair up the $n-1$ letters $C_i$ to form complementary pairs. If $n$ is even
then we can pair up $n-2$ letters $C_i$ to form complementary pairs and we get
left with a single element $C_i$ with $i$ equal to $(n-2)/2$. Note that
if all other elements have been eliminated from $Q(\cable{\alpha}{n})$ then
this final $C_i$ is a core element.
From this discussion we have the following results.
\begin{prop}
For any virtual string $\Gamma$ and any natural number $n$,
\begin{equation*}
\rho(\cable{\Gamma}{n}) \leq n^2\rho(\Gamma) + \delta_n
\end{equation*}
where $\delta_n$ is $1$ if $n$ is even and $0$ if $n$ is odd.
\end{prop}
\begin{prop}
If $\rho(\Gamma)$ is $0$ then $\rho(\cable{\Gamma}{n})$ is also $0$.
\end{prop}
\bibliography{mrabbrev,ccc}
\bibliographystyle{hamsplain}
\end{document}